\documentclass[a4paper]{article}
\usepackage{hyperref}
\usepackage{amsmath}
\usepackage{amsfonts}
\usepackage{amsthm}
\usepackage{indentfirst}
\usepackage{geometry}

\newtheorem{theorem}{Theorem}[section]

\newtheorem{proposition}[theorem]{Proposition}

\catcode`@=11
\@addtoreset{equation}{section}
\catcode`@=12

\begin{document}

\begin{center}
{\Large  Finite energy solutions and critical conditions of
nonlinear equations in $R^n$}
\end{center}

\vskip 5mm

\begin{center}

{\sc Yutian Lei} \\
\vskip 3mm
Institute of Mathematics\\
School of Mathematical Sciences\\
Nanjing Normal University\\
Nanjing, 210023, China\\
email: leiyutian@njnu.edu.cn

\end{center}

\vskip 5mm {\leftskip5mm\rightskip5mm \normalsize
\noindent{\bf{Abstract}}
This paper is concerned with the critical conditions of nonlinear
elliptic equations with weights and the corresponding integral equations with Riesz
potentials and Bessel potentials. We show that the equations and some energy functionals
are invariant under the scaling transformation if and only if the critical
conditions hold. In addition, the Pohozaev identity shows that those critical
conditions are the necessary and sufficient conditions for
existence of the finite energy positive solutions or weak solutions.
Finally, we discuss respectively the existence of
the negative solutions of the $k$-Hessian equations in the subcritical case,
critical case and supercritical case. Here the Serrin exponent and the
critical exponent play key roles.


}

\section{Introduction}  

In this paper, we consider the relation between the critical
conditions and the finite energy solutions for several semilinear,
quasilinear and fully nonlinear elliptic equations.

If $n \geq 3$, and $u$ belongs to the homogeneous
Sobolev space $\mathcal{D}^{1,2}(R^n)$ such that the Sobolev inequality holds
\begin{equation} \label{svi}
\|u\|_{L^{q+1}(R^n)}^2 \leq C\|\nabla u\|_{L^2(R^n)}^2,
\end{equation}
then $q=\frac{n+2}{n-2}$. In fact, after
the scaling transformation
\begin{equation} \label{scal}
u_\mu(x):=\mu^{\sigma}u(\mu x), \quad \mu>0,
\end{equation}
by (\ref{svi}) we can see that
$
\|u_\mu\|_{L^{q+1}(R^n)}^2 \leq C\mu^{n-2-\frac{2n}{q+1}}\|\nabla
u_\mu\|_{L^2(R^n)}^2.
$
Since $u_\mu$ also satisfies (\ref{svi}), $q$ must be equal to
$\frac{n+2}{n-2}$.

The Euler-Lagrange equation which the extremal function of (\ref{svi})
satisfies is the Lane-Emden equation
\begin{equation} \label{el2}
-\Delta u=u^{q}, \quad u>0 ~in~ R^n.
\end{equation}
Eq. (\ref{el2}) and the energy $\|u\|_{L^{q+1}(R^n)}$ are
invariant under the scaling transformation if and only if
$q=\frac{n+2}{n-2}$. In fact, $u_\mu$
solves (\ref{el2}) implies $\sigma=\frac{2}{q-1}$.
The energy $\|u_\mu\|_{L^{q+1}(R^n)}=\|u\|_{L^{q+1}(R^n)}$
implies $\sigma=\frac{n}{q+1}$. Thus, $q=\frac{n+2}{n-2}$.
On the contrary, if $q=\frac{n+2}{n-2}$, (\ref{el2}) is
invariant under the conformal transformation.

The critical exponent $\frac{n+2}{n-2}$ plays the key roles on the
existence and nonexistence of this Lane-Emden equation. We
refer to \cite{GS} by Gidas and Sprunk for details.

The solution $u$ is called a finite energy solution if $u \in L^{q+1}(R^n)$.
It is not difficult to verify that $u
\in C^\infty(R^n) \cap \mathcal{D}^{1,2}(R^n)$ is equivalent to $u
\in C^\infty(R^n) \cap L^{q+1}(R^n)$. In addition, $\|\nabla
u\|_{L^2(R^n)}=\|u\|_{L^{q+1}(R^n)}$.
The classification result by Chen and Li \cite{CL-Duke} shows that
(\ref{el2}) has the finite energy solutions if and only if
$q=\frac{n+2}{n-2}$. On the contrary, all the solutions of (\ref{el2})
in the critical case are the finite energy solution.

Next, we consider the Lane-Emden system
\begin{equation} \label{le2}
 \left \{
   \begin{array}{l}
      -\Delta u=v^{q_2}, ~u,v>0~in~R^n,\\
      -\Delta v=u^{q_1}, ~q_1,q_2>1.
   \end{array}
   \right.
\end{equation}
Instead of the critical exponent $q=\frac{n+2}{n-2}$,
the critical condition which $q_1,q_2$ satisfy
is
\begin{equation} \label{cc}
\frac{1}{q_1+1}+\frac{1}{q_2+1}=1-\frac{2}{n}.
\end{equation}
It also comes into play in the study of the existence for (\ref{le2}). When
$\frac{1}{q_1+1}+\frac{1}{q_2+1} \leq \frac{n-2}{n}$, the existence of
classical positive solutions had been verified by Mitidieri, Serrin and Zou
(cf. \cite{Mi}, \cite{SZ-1998}). Nonexistence of
positive solution is still open when
$\frac{1}{q_1+1}+\frac{1}{q_2+1}>\frac{n-2}{n}$ except for the
case of $n \leq 4$ (cf. \cite{Souplet}). This Liouville type
property is the well known Lane-Emden conjecture.

All the results above can be generalize to an
integral equation involving the Reisz potential
(cf. \cite{CLO} and \cite{YLi})
\begin{equation} \label{hls}
    u(x)= \displaystyle\int_{R^n}\frac{u^{q}(y)dy}{|x-y|^{n-\alpha}},
\end{equation}
where $n\geq 3$, $\alpha \in (0,n)$, $q>0$.
It is also invariant under the conformal transformation.
For the weighted equations, such as the Hardy-Sobolev type, the
Caffarelli-Kohn-Nirenberg type and the weighted
Hardy-Littlewood-Sobolev type, the invariant is still true under the scaling
transformation. However,
the invariant is absent under the translation. On the other hand,
for the equations involving the Bessel potentials,
the invariant is true under the translation, but false under the scaling.

The following system corresponding (\ref{hls}) is related to the study of the extremal
functions of the Hardy-Littlewood-Sobolev inequality (cf. \cite{L})
\begin{equation} \label{Hls}
 \left \{
   \begin{array}{l}
   u(x)= \displaystyle\int_{R^n}\frac{v^{q_2}(y)dy}{|x-y|^{n-\alpha}}\\
   v(x)= \displaystyle\int_{R^n}\frac{u^{q_1}(y)dy}{|x-y|^{n-\alpha}},
   \end{array}
   \right.          
\end{equation}
Recently, \cite{LeiLi} shows that the Euler-Lagrange system (\ref{Hls})
and energy functionals $\|u\|_{L^{q_1+1}(R^n)}$ and $\|v\|_{L^{q_2+1}(R^n)}$
are invariant under the scaling transformation
\begin{equation} \label{scalsy}
u_\mu(x)=\mu^{\sigma_1}u(\mu x), \quad
v_\mu(x)=\mu^{\sigma_2}v(\mu x),
\end{equation}
if and only if the condition $\frac{1}{q_1}+\frac{1}{q_2}=1-\frac{\alpha}{n}$
holds. In addition, (\ref{Hls}) has finite energy solutions if and only
if $q_1,q_2$ satisfy such a critical condition. However, it is open
whether or not all positive solutions in the critical
case are the finite energy solutions.

In this paper, we always assume $n \geq 3$, $q,q_1,q_2>1$. We expect to generalize
the argument above to other nonlinear equations, including higher order and
fractional order semilinear equations, p-Laplace equation and system, and
k-Hessian equations.

In Section 2, we point out the relation between finite energy
solutions and weak solutions, and prove that the critical conditions are the
necessary and sufficient conditions for the existence of the finite energy
solutions of the equations involving the Riesz potentials. For the equations
involving the Bessel potentials, we prove that subcritical conditions are the
necessary conditions for the existence of finite energy solutions. This shows
the corresponding energy functional has no minimizer in critical case. We
present the minimum by the least energy whose Euler-Lagrange equation involves
the Riesz potential (cf. Theorem \ref{th3.3}).

In Section 3, we study the Caffarelli-Kohn-Nirenberg
type p-Laplacian equation and system, and surprisingly find that the critical condition of
the system is degenerate to two simple cases when we investigate the invariant
of the system and the energy functionals under the scaling transformation:
either $p=2$, or the system is reduced to a single equation (cf. Theorem \ref{th4.7}).
Unfortunately, the system has no variational structure, and hence we cannot use the Pohozaev
identity to verify whether or not there exists a nondegenerate critical condition determining
the existence of the finite energy solutions.

Finally, in Section 4, we study a k-Hessian equation. We present the nonexistence of negative
solution when the exponent is smaller than the Serrin exponent. In addition,
we find a radial solution with slow decay rate in the supercritical case
(cf. Theorem \ref{th4.10}), and
another radial solution with fast decay rate in the critical case. Based on this
result, we prove the critical condition is the necessary and sufficient condition
for the existence of finite energy solutions (cf. Theorem \ref{th4.5}).

\section{Semilinear equations}

\subsection{Hardy-Sobolev type equations}

We search the values of $q$ such that the classical Hardy-Sobolev
inequality holds
\begin{equation} \label{hsi}
(\int_{R^n}|x|^{-t}u^{q+1}dx)^{\frac{2}{q+1}} \leq
C_n\int_{R^n}|\nabla u|^2dx,
\end{equation}
for all $u \in \mathcal{D}^{1,2}(R^n)$. Here $n \geq 3$, $t \in
(0,2)$.

In order to verify this inequality still holds for $u_\mu$ (cf.
(\ref{scal})), we have
$$\begin{array}{ll}
&(\displaystyle\int_{R^n}|x|^{-t}u_\mu^{q+1}(x)dx)^{\frac{1}{q+1}}
=\mu^{\sigma-\frac{n-t}{q+1}}(\displaystyle\int_{R^n}
|y|^{-t}u^{q+1}(y)dy)^{\frac{1}{q+1}}\\[3mm]
&\leq
C_n\mu^{\sigma-\frac{n-t}{q+1}}(\displaystyle\int_{R^n}|\nabla
u(y)|^2dy)^{1/2} \leq C_n \mu^{\frac{n-2}{2}-\frac{n-t}{q+1}}
(\displaystyle\int_{R^n}|\nabla u_\mu(x)|^2dx)^{1/2},
\end{array}
$$
and hence $\frac{n-2}{2}-\frac{n-t}{q+1}=0$, which implies
$q=\frac{n+2-2t}{n-2}$.

The extremal functions in $\mathcal{D}^{1,2}(R^n) \setminus \{0\}$ of
(\ref{hsi}) can be obtained by investigating the functional
$$
E(u)=\|\nabla
u\|_{L^2(R^n)}^2(\displaystyle\int_{R^n}|x|^{-t}
u^{q+1}dx)^{\frac{-2}{q+1}}.
$$

Consider the Euler-Lagrange equation
\begin{equation} \label{hsel}
-\Delta u=|x|^{-t}u^q, \quad u>0~in ~R^n.
\end{equation}
In view of
$
-\Delta u_\mu(x)=-\mu^{\sigma+2}\Delta u(\mu
x)=\mu^{\sigma+2-t-q\sigma}|x|^{-t}u_\mu^q(x),
$
we can see that $\sigma=\frac{2-t}{q-1}$ if and only if $u_\mu$
solves (\ref{hsel}). In addition, noting
\begin{equation} \label{energy}
\int_{R^n}|x|^{-t}u_\mu^{q+1}(x)dx
=\mu^{\sigma(q+1)}\displaystyle\int_{R^n}|x|^{-t}u^{q+1}(\mu
x)dx
=\mu^{\sigma(q+1)-n+t}\displaystyle\int_{R^n}|y|^{-t}u^{q+1}(y)dy,
\end{equation}
we can see that $\sigma=\frac{n-t}{q+1}$ if and only if the energy
$\||x|^{\frac{-t}{q+1}}u\|_{L^{q+1}(R^n)}$ is invariant under the
scaling (\ref{scal}). Eliminating $\sigma$ we also obtain that $q$
is the critical exponent $\frac{n+2-2t}{n-2}$.

\begin{theorem} \label{th2.3}
Eq. (\ref{hsel}) has a weak solution in $\mathcal{D}^{1,2}(R^n)$ if and only if
$q=\frac{n+2-2t}{n-2}$.
\end{theorem}

\begin{proof}
In fact, if $q=\frac{n+2-2t}{n-2}$, the radial function
\begin{equation} \label{clsf}
u(x)=c(\frac{d}{d^2+|x|^{2-t}})^{\frac{n-2}{2-t}}
\end{equation}
belongs to $\mathcal{D}^{1,2}(R^n)$ and solves (\ref{hsel}). Here $c,d>0$.

On the contrary, since the weak solution is a critical point of the
functional $E(u)$, we have the Pohozaev identity
$
[\frac{d}{d\mu}E(u(\frac{x}{\mu}))]_{\mu=1}=0.
$
Noting
$
E(u(\frac{x}{\mu}))=\mu^{n-2-\frac{2(n-t)}{q+1}}E(u(x)),
$
we get $q=\frac{n+2-2t}{n-2}$.
\end{proof}

Clearly, if $q=\frac{n+2-2t}{n-2}$, then $u \in
\mathcal{D}^{1,2}(R^n)$ implies $|x|^{\frac{-t}{q+1}}u \in
L^{q+1}(R^n)$ by the Hardy-Sobolev inequality. A natural question
is, for a general exponent $q$, when the energy
$\||x|^{\frac{-t}{q+1}}u\|_{L^{q+1}(R^n)}$ is finite.

\begin{proposition} \label{prop2.2}
(1) If $u \in C^2(R^n) \cap \mathcal{D}^{1,2}(R^n)$ solves (\ref{hsel}), then
$\||x|^{\frac{-t}{q+1}}u\|_{L^{q+1}(R^n)}^{q+1}<\infty$. In addition,
$\||x|^{\frac{-t}{q+1}}u\|_{L^{q+1}(R^n)}^{q+1}=\|\nabla u\|_{L^{2}(R^n)}^2$.

(2) Assume $u \in C^2(R^n)$ solves (\ref{hsel}) and
$\int_{R^n}|x|^{-t}u^{q+1}dx<\infty$. If $u \in L^{\frac{2n}{n-2}}(R^n)$, then $u
\in \mathcal{D}^{1,2}(R^n)$, and $\int_{R^n}|\nabla u|^2dx
=\int_{R^n}|x|^{-t}u^{q+1}dx$.
\end{proposition}

\begin{proof}
(1) Multiplying (\ref{hsel}) by $u$ and integrating on
$B_R(0)$, we have
\begin{equation} \label{fushi}
\int_{B_R(0)}|\nabla u|^2dx-\int_{\partial B_R(0)}u
\partial_{\nu}uds =\int_{B_R(0)}|x|^{-t}u^{q+1}dx.
\end{equation}
By virtue of $u \in \mathcal{D}^{1,2}(R^n)$, there exists $R=R_j \to
\infty$ such that
\begin{equation} \label{changwei}
R\int_{\partial B_R(0)}(|\nabla u|^2+u^{\frac{2n}{n-2}})ds \to 0.
\end{equation}
By this result and the H\"older inequality, we get
$$
|\int_{\partial B_R(0)}u \partial_{\nu}uds| \leq
\|\partial_{\nu}u\|_{L^2(\partial
B_R(0))}\|u\|_{L^{\frac{2n}{n-2}}(\partial B_R(0))} |\partial
B_R(0)|^{\frac{1}{2}-\frac{n-2}{2n}} \to 0
$$
when $R=R_j \to \infty$. Inserting this result into (\ref{fushi}),
we can see $\int_{R^n}\frac{u^{q+1}dx}{|x|^t} =\int_{R^n}|\nabla
u|^2dx$.

(2)
On the contrary, take smooth function $\zeta(x)$ satisfying
$$
\left \{
   \begin{array}{lll}
&\zeta(x)=1, \quad &for~ |x| \leq 1;\\
&\zeta(x) \in [0,1], \quad &for~ |x| \in [1,2];\\
&\zeta(x)=0, \quad &for~ |x| \geq 2.
   \end{array}
   \right.
$$
Define the cut-off function
\begin{equation} \label{cut}
\zeta_R(x)=\zeta(\frac{x}{R}).
\end{equation}
Multiplying (\ref{hsel}) by $u\zeta_R^2$ and integrating on
$B_{2R}(0)$, we have
\begin{equation} \label{heyue}
\int_{B_{2R}(0)}|\nabla u|^2 \zeta_R^2dx
=2\int_{B_{2R}(0)}u\zeta_R\nabla u \nabla \zeta_R dx +\int_{B_{2R}(0)}
|x|^{-t}u^{q+1} \zeta_R^2dx.
\end{equation}
Clearly, there exists $C>0$ which is independent of $R$, such that
$$
|\int_{B_{2R}(0)}u\zeta_R\nabla u \nabla \zeta_R dx| \leq
\frac{1}{4}\displaystyle\int_{B_{2R}(0)}|\nabla u|^2 \zeta_R^2dx
+C\int_{B_{2R}(0)}u^2|\nabla \zeta_R|^2dx.
$$
If $u \in L^{\frac{2n}{n-2}}(R^n)$, there holds
$$
\int_{B_{2R}(0)}u^2|\nabla \zeta_R|^2dx \leq
\frac{C}{R^2}(\int_{B_{2R}(0)}u^{\frac{2n}{n-2}}dx)^{1-\frac{2}{n}}
|B_{2R}(0)|^{\frac{2}{n}} \leq C.
$$
Inserting these results into (\ref{heyue}) and noting
$\int_{R^n}|x|^{-t}u^{q+1}dx<\infty$, we get
$
\int_{B_{2R}(0)}|\nabla u|^2 \zeta_R^2dx \leq C,
$
where $C>0$ is independent of $R$. Letting $R \to \infty$, we have
$\nabla u\in L^2(R^n)$, and hence $u \in \mathcal{D}^{1,2}(R^n)$.
Thus, (\ref{changwei}) still
holds, and from (\ref{fushi}) we also deduce $\int_{R^n}|\nabla
u|^2dx =\int_{R^n}|x|^{-t}u^{q+1}dx$.
\end{proof}

The positive solution $u \in C^2(R^n)$ is called a finite energy
solution of (\ref{hsel}), if
$$
\int_{R^n}|x|^{-t}u^{q+1}dx<\infty.
$$
In the critical case $q=\frac{n+2-2t}{n-2}$, (\ref{clsf}) is a finite
energy solution. On the contrary, if (\ref{hsel}) has a finite energy
solution with $q\leq \frac{n+2-2t}{n-2}$, then Proposition \ref{prop2.2}
and Theorem \ref{th2.3} imply $q=\frac{n+2-2t}{n-2}$.

\vskip 5mm

The argument above can be generalized to the higher order system
involving two coupled equations
\begin{equation} \label{hssy}
 \left \{
   \begin{array}{l}
      (-\Delta)^{l} u=|x|^{-t}v^{q_2}, ~u>0~in~R^n,\\
      (-\Delta)^{l} v=|x|^{-t}u^{q_1}, ~v>0~in~R^n.
   \end{array}
   \right.
\end{equation}
Here $l\in [1,n/2)$ is an integer.

\begin{proposition} \label{prop2.4}
Under the scaling transformation (\ref{scalsy}), the equation (\ref{hssy}) and
the energy functionals
$\||x|^{\frac{-t}{q_1+1}}u\|_{L^{q_1+1}(R^n)}$ and
$\||x|^{\frac{-t}{q_2+1}}v\|_{L^{q_2+1}(R^n)}$ are invariant, if and only if
$q_1$ and $q_2$ satisfy the critical condition
\begin{equation} \label{criticalcondition}
\frac{1}{q_1+1}+\frac{1}{q_2+1}=\frac{n-2l}{n-t}.
\end{equation}
\end{proposition}

\begin{proof}
Set $y=\mu x$. By (\ref{scalsy}) and (\ref{hssy}), we have
$$
(-\Delta)^lu_\mu(x)=\mu^{\sigma_1+2l}(-\Delta)^lu(y)
=\mu^{\sigma_1+2l}|y|^{-t}v^{q_2}(y)
=\mu^{\sigma_1+2l-t-q_2\sigma_2}|x|^{-t}v_\mu^{q_2}(x).
$$
Eq. (\ref{hssy}) is invariant under the scaling (\ref{scalsy})
implies $\sigma_1+2l-t-q_2\sigma_2=0$ and
$\sigma_2+2l-t-q_1\sigma_1=0$. By the same derivation of
(\ref{energy}), we also obtain $\sigma_1(q_1+1)-n+t=0$ and
$\sigma_2(q_2+1)-n+t=0$ by the invariant of
$\||x|^{\frac{-t}{q_1+1}}u\|_{L^{q_1+1}(R^n)}$ and
$\||x|^{\frac{-t}{q_2+1}}v\|_{L^{q_2+1}(R^n)}$. Eliminating
$\sigma_1$ and $\sigma_2$, we can see
$
\frac{q_1q_2-1}{(q_1+1)(q_2+1)}=\frac{2l-t}{n-t}.
$
In view of $q_1q_2-1=(q_1+1)(q_2+1)-(q_1+1)-(q_2+1)$, it follows
(\ref{criticalcondition}).

On the contrary, the calculation above still implies the sufficiency.
\end{proof}

It seems difficult to generalized this process to the system
involving $m$ equations with $m \geq 3$. How to obtain the
critical conditions of the system involving $m$ equations is an
interesting problem.

The classical solutions $u,v$ of (\ref{hssy}) are called
{\it finite energy solutions} if
$$
\||x|^{\frac{-t}{q_1+1}}u\|_{L^{q_1+1}(R^n)}<\infty, \quad
\||x|^{\frac{-t}{q_2+1}}v\|_{L^{q_2+1}(R^n)}<\infty.
$$

\begin{theorem} \label{th2.5}
Eq. (\ref{hssy}) has finite energy solutions
if and only if (\ref{criticalcondition}) holds.
\end{theorem}

\begin{proof}
When $\frac{1}{q_1+1}+\frac{1}{q_2+1}=\frac{n-2l+\beta_1+\beta_2}{n}$,
Lieb \cite{L} obtained a pair of extremal functions $(U,V) \in
L^{q_1+1}(R^n) \times L^{q_2+1}(R^n)$ of the weighted
Hardy-Littlewood-Sobolev inequality, which
solves the integral system
$$
 \left \{
   \begin{array}{l}
      U(x) = \displaystyle\frac{1}{|x|^{\beta_1}}\int_{R^{n}} \frac{V^{q_2}(y)}{|y|^{\beta_2}|x-y|^{n-2l}}  dy\\
      V(x) = \displaystyle\frac{1}{|x|^{\beta_2}}\int_{R^{n}} \frac{U^{q_1}(y)}{|y|^{\beta_1}|x-y|^{n-2l}}
      dy.
   \end{array}
   \right.
$$
If (\ref{criticalcondition}) is true, we can choose $\beta_1$ and
$\beta_2$ satisfying $\beta_1(q_1+1)=\beta_2(q_2+1)=t$. Taking
$u(x)=|x|^{\beta_1}U(x)$ and $v(x)=|x|^{\beta_2}V(x)$, we can see
that $(u,v)$ solves
\begin{equation} \label{kunsha}
 \left \{
   \begin{array}{l}
      u(x) = \displaystyle\int_{R^{n}} \frac{v^{q_2}(y)}{|y|^{t}|x-y|^{n-2l}}  dy\\
      v(x) = \displaystyle\int_{R^{n}} \frac{u^{q_1}(y)}{|y|^{t}|x-y|^{n-2l}}
      dy.
   \end{array}
   \right.
\end{equation}
In addition, according to the radial symmetry and integrability results (cf.
\cite {CJ}, \cite{CJ2}) and the asymptotic behavior of $(U,V)$ (cf.
\cite{LLM}), $u$ and $v$ are finite energy solutions. By the properties of the
Riesz potentials, it follows that $(u,v)$ solves (\ref{hssy})
from (\ref{kunsha}).

On the contrary, according to the equivalence results in \cite{CL-cpaa},
the classical solutions of (\ref{hssy}) also satisfy (\ref{kunsha}).
In the following, we use the Pohozaev identity of integral forms
introduced in \cite{CDM} to deduce (\ref{criticalcondition}).

By (\ref{kunsha}) we have
\begin{equation} \label{equa}
\begin{array}{ll}
&\displaystyle\int_{R^n}\frac{u^{q_1+1}(x)}{|x|^t}dx=\int_{R^n}\frac{u^{q_1}(x)}{|x|^t}
\int_{R^{n}} \frac{v^{q_2}(y)}{|y|^{t}|x-y|^{n-2l}}  dydx\\[3mm]
&=\displaystyle\int_{R^n}\frac{v^{q_2}(y)}{|y|^{t}}
\int_{R^{n}} \frac{u^{q_1}(x)}{|x-y|^{n-2l}}  dx dy
=\int_{R^n}\frac{v^{q_2+1}(y)}{|y|^t}dy.
\end{array}
\end{equation}
For $\mu>0$, from (\ref{kunsha}) it follows
$$
x \cdot \nabla u(x)=\frac{d}{d\mu}u(\mu x)|_{\mu=1}
=(2l-t)u(x)+\int_{R^n}\frac{z \cdot \nabla v^{q_2}(z)dz}{|z|^t|x-z|^{n-2l}}
$$
Multiplying by $|x|^{-t}u^{q_1}(x)$ and integrating on $R^n$, we get
\begin{equation} \label{duanwu}
\begin{array}{ll}
&\displaystyle\int_{R^n}\frac{u^{q_1}(x)}{|x|^t}x \cdot \nabla u(x)dx
-(2l-t)\int_{R^n}\frac{u^{q_1+1}(x)}{|x|^t}dx\\[3mm]
&=\displaystyle\int_{R^n}\frac{u^{q_1}(x)}{|x|^t}
\int_{R^n}\frac{z \cdot \nabla v^{q_2}(z)dz}{|z|^t|x-z|^{n-2l}}dx\\[3mm]
&=\displaystyle\int_{R^n}\frac{z \cdot \nabla v^{q_2}(z)}{|z|^t}
\int_{R^n}\frac{u^{q_1}(x)}{|x|^t|z-x|^{n-2l}}dxdz
=\int_{R^n}\frac{z \cdot \nabla v^{q_2}(z)}{|z|^t}v(z)dz.
\end{array}
\end{equation}
Since $u,v$ are finite energy solutions, we can find $R=R_j \to \infty$
such that
$$
R^{1-t}\int_{\partial B_R(0)}(u^{q_1+1}+v^{q_2+1})ds \to 0.
$$
Thus, integrating (\ref{duanwu}) by parts yields
$$
(\frac{t-n}{q_1+1}-2l+t)\int_{R^n}\frac{u^{q_1+1}}{|x|^t}dx
=\frac{q_2(t-n)}{q_2+1}\int_{R^n}\frac{v^{q_2+1}}{|x|^t}dx.
$$
Combining with (\ref{equa}) we obtain (\ref{criticalcondition}).
\end{proof}

\paragraph{Remark 2.1.}
We have two direct corollaries:

(1) If $l=1$, then (\ref{hssy}) has finite energy solutions
if and only if $\frac{1}{q_1+1}+\frac{1}{q_2+1}=\frac{n-2}{n-t}$.

(2) If $q_1=q_2$ and $u=v$, then (\ref{hssy}) has finite energy solution
if and only if $q_1=\frac{n+2l-2t}{n-2l}$.

Noting the conditions in Theorems \ref{th2.3} and \ref{th2.5}, it is
convenient for us to discuss the finite energy solutions for integral equations,
and the weak solutions in $\mathcal{D}^{l,2}(R^n)$ for the differential equations
respectively.

\subsection{WHLS type integral system}

Let $1 < r, s < \infty$, $0 <\lambda < n$, $\beta_1+\beta_2 \geq 0$
and $\beta_1+\beta_2 \leq \alpha$. The weighted Hardy-Littlewood-Sobolev
(WHLS) inequality states that (cf. \cite{SW2})
\begin{equation} \label{Whls}
   \left|\int_{R^n}\int_{R^n}
   \frac{f(x)g(y)}{|x|^{\beta_1}|x-y|^{n-\alpha}|y|^{\beta_2}}dxdy
   \right| \le C_{\beta_1,\beta_2,s,\alpha,n }\|f\|_r \|g\|_s
\end{equation}
where $1 - \frac{1}{r} - \frac{n-\alpha}{n} < \frac{\beta_1}{n} < 1 -
\frac{1}{r}$.
If the inequality (\ref{Whls}) still holds for the scaling functions
$f_\mu$ and $g_\mu$ (cf. (\ref{scalsy})), then we can deduce
\begin{equation} \label{B1}
\frac{1}{r} + \frac{1}{s} +
\frac{n-\alpha + \beta_1+\beta_2}{n} = 2 .
\end{equation}

In order to obtain the sharp constant in the WHLS inequality
(\ref{Whls}), we maximize the functional
$$
    J(f,g) = (\|f\|_r \|g\|_s)^{-1} \int_{R^n}\int_{R^n}\frac{f(x)g(y)}{|x|^{\beta_1}|x-y|^{n-\alpha}|y|^{\beta_2}} dx dy.
$$
The corresponding
Euler-Lagrange equations are the following integral system:
$$
 \left \{
   \begin{array}{l}
      \lambda_1 r {f(x)}^{r-1} =\displaystyle \frac{1}{|x|^{\beta_1}}\int_{R^{n}} \frac{g(y)}{|y|^{\beta_2}|x-y|^{n-\alpha}}  dy,\\
      \lambda_2 s {g(x)}^{s-1} =\displaystyle \frac{1}{|x|^{\beta_2}}\int_{R^{n}} \frac{f(y)}{|y|^{\beta_1}|x-y|^{n-\alpha}}  dy,
   \end{array}
   \right.
$$
where $f, g \geq 0$, and $\lambda_1 r =\lambda_2 s =J(f,g)$.

If $f \in L^r(R^n)$ and $g \in L^s(R^n)$, then the Pohozaev identity
$\frac{dJ(f(x\mu^{-1}),g(x\mu^{-1}))}{d\mu}|_{\mu=1}=0$ still implies (\ref{B1}).
In fact, for $\mu > 0$,
$$
J(f(\frac{x}{\mu}),g(\frac{x}{\mu}))=\mu^{2n-(n-\alpha+\beta_1+\beta_2)-\frac{n}{r}
-\frac{n}{s}}J(f(x),g(x)).
$$
Thus, $\frac{dJ(f(x/\mu),g(x/\mu))}{d\mu}|_{\mu=1}=0$ leads to (\ref{B1}).

Let $u = c_1 f^{r-1}$, $v = c_2 g^{s-1}$, and $q_1 =\frac{1}
{r-1}$, $q_2 = \frac{1}{s-1}$. Choosing suitable $c_1$ and $c_2$,
we obtain that the corresponding
Euler-Lagrange equations are the following integral system
\begin{equation}
 \left \{
   \begin{array}{l}
      u(x) = \displaystyle\frac{1}{|x|^{\beta_1}}\int_{R^{n}} \frac{v^{q_2}(y)}{|y|^{\beta_2}|x-y|^{n-\alpha}}  dy\\
      v(x) = \displaystyle\frac{1}{|x|^{\beta_2}}\int_{R^{n}} \frac{u^{q_1}(y)}{|y|^{\beta_1}|x-y|^{n-\alpha}}  dy
   \end{array}
   \right.\label{ELWHLS}          
 \end{equation}
 where $\beta_1+\beta_2 \leq \alpha$, and
     \begin{equation} \label{YYY}
     \left \{
   \begin{array}{l}
   u, v \geq 0, \;\; 0 <q_1, q_2<\infty,\;
    0 <\alpha < n, \; \beta_1 \geq 0,\beta_2  \geq 0,\; \\
  \displaystyle \frac{\beta_1}{n}  <
  \frac{1}{q_1+1}<\frac{n-\alpha+\beta_1}{n},\;  \frac{\beta_2}{n}  <
  \frac{1}{q_2+1}<\frac{n-\alpha+\beta_2}{n}.
\end{array}
   \right.
   \end{equation}

The equation (\ref{ELWHLS}) and the energy functionals
$\|u\|_{L^{q_1+1}(R^n)}$, $\|v\|_{L^{q_2+1}(R^n)}$ are invariant under the scaling transformation
(\ref{scalsy}), if and only if
\begin{equation} \label{ELWHLScond2}
\frac{1}{q_1+1}+\frac{1}{q_2+1}=\frac{n-\alpha+\beta_1+\beta_2}{n}.
\end{equation}
Clearly, (\ref{ELWHLScond2}) is equivalent to (\ref{B1}).

Since $(f,g) \in L^r(R^n) \times L^s(R^n)$ implies $(u,v) \in
L^{q_1+1}(R^n) \times L^{q_2+1}(R^n)$, we call such a pair of solutions $(u,v)$
the {\it finite energy solutions}. In addition, (\ref{ELWHLScond2}) is called
the {\it critical condition}.

\begin{theorem} \label{th2.2}
Eq. (\ref{ELWHLS}) has the finite energy
positive solutions in $C_{loc}^1(R^n \setminus \{0\})$
if and only if $p,q$ satisfy the critical
condition (\ref{ELWHLScond2}).
\end{theorem}


\begin{proof}

{\it Sufficiency.}

According to \cite{L}, the existence of the extremal functions of
the WHLS inequality implies our conclusion. In fact, those extremal functions
are finite energy solutions. By a regularity lifting process, the extremal
functions also belong to $C_{loc}^1(R^n \setminus \{0\})$.

{\it Necessity.}

Denote $n-\alpha+\beta_1+\beta_2$ by $\bar{\lambda}$.
For $x \neq 0$ and $\mu>0$, we have
$$
u(\mu x)=\int_{R^n}\frac{v^{q_2}(y)dy}{|\mu x|^{\beta_1}|\mu
x-y|^{n-\alpha}|y|^{\beta_2}}
=\mu^{n-\bar{\lambda}}\int_{R^n}\frac{v^{q_2}(\mu
z)dz}{|x|^{\beta_1}|x-z|^{n-\alpha}|z|^{\beta_2}}.
$$
Differentiating with respect to $\mu$ and then letting $\mu=1$, we get
\begin{equation} \label{poho}
x\cdot\nabla u(x)=(n-\bar{\lambda})u
+\lim_{d \to 0}\int_{R^n\setminus B_d(0)}\frac{z\cdot\nabla v^{q_2}(z)dz}
{|x|^{\beta_1}|x-z|^{n-\alpha}|z|^{\beta_2}}.
\end{equation}
Multiplying by $u^{q_1}$ and integrating on $R^n \setminus B_d(0)$,
we have
\begin{equation} \label{shuix}
\begin{array}{ll}
&\displaystyle\lim_{d \to 0} \int_{R^n \setminus B_d(0)}u^{q_1}(x\cdot\nabla u(x))dx
=(n-\bar{\lambda})\int_{R^n}u^{q_1+1}(x)dx\\[3mm]
&\quad +\displaystyle\lim_{d \to 0}\int_{R^n\setminus B_d(0)}u^{q_1}(x)
\int_{R^n\setminus B_d(0)}\frac{z\cdot\nabla v^{q_2}(z)dz}
{|x|^{\beta_1}|x-z|^{n-\alpha}|z|^{\beta_2}}dx.
\end{array}
\end{equation}
Integrating by parts, we get
$$\begin{array}{ll}
K_1&:=\displaystyle\lim_{d \to 0} \int_{R^n \setminus B_d(0)}u^{q_1}(x\cdot\nabla u(x))dx
=\lim_{d \to 0} \frac{1}{q_1+1} \int_{R^n \setminus B_d(0)}(x\cdot\nabla u^{q_1+1}(x))dx\\[3mm]
&=\displaystyle\lim_{r \to \infty} \frac{r}{q_1+1}\int_{\partial B_r(0)}u^{q_1+1}(x)ds
-\lim_{d \to 0} \frac{d}{q_1+1}\int_{\partial B_d(0)}u^{q_1+1}(x)dx\\[3mm]
&\quad -\displaystyle\frac{n}{q_1+1}\int_{R^n}u^{q_1+1}(x)dx.
\end{array}
$$
In view of $u \in L^{q_1+1}(R^n)$, we can find $r=r_j \to \infty$ and
$d=d_m \to \infty$ such that
$$
\lim_{r \to \infty} r\int_{\partial B_r(0)}u^{q_1+1}(x)ds
=\lim_{d \to 0} d\int_{\partial B_d(0)}u^{q_1+1}(x)dx=0,
$$
and hence
$$
K_1=-\frac{n}{q_1+1}\int_{R^n}u^{q_1+1}(x)dx.
$$
Using the Fubini theorem, we have
$$\begin{array}{ll}
K_2&:=\displaystyle\lim_{d \to 0}\int_{R^n\setminus B_d(0)}u^{q_1}(x)
\int_{R^n\setminus B_d(0)}\frac{z\cdot\nabla v^{q_2}(z)dz}
{|x|^{\beta_1}|x-z|^{n-\alpha}|z|^{\beta_2}}dx\\[3mm]
&=\displaystyle\lim_{d \to 0}\int_{R^n\setminus B_d(0)}z\cdot\nabla v^{q_2}(z)
\int_{R^n\setminus B_d(0)}\frac{u^{q_1}(x)dx}
{|z|^{\beta_2}|z-x|^{n-\alpha}|x|^{\beta_1}} dz\\[3mm]
&=\displaystyle\lim_{d \to 0}\int_{R^n\setminus B_d(0)}(z\cdot\nabla v^{q_2}(z))v(z)dz.
\end{array}
$$
Similar to the calculation of $K_1$, we also obtain
$$
K_2=-\frac{q_2n}{q_2+1}\int_{R^n}v^{q_2+1}(z)dz.
$$
Inserting $K_1$ and $K_2$ into (\ref{shuix}), we have
$$
-\frac{n}{q_1+1}\int_{R^n}u^{q_1+1}(x)dx=(n-\bar{\lambda})\int_{R^n}u^{q_1+1}(x)dx
-\frac{q_2n}{q_2+1}\int_{R^n}v^{q_2+1}(z)dz.
$$

By (\ref{ELWHLS}) and the Fubini theorem, we also have
$$\begin{array}{ll}
&\quad\displaystyle\int_{R^n}u^{q_1+1}(x)dx=\int_{R^n}u^{q_1}(x)u(x)dx
=\int_{R^n}u^{q_1}(x)\int_{R^{n}} \frac{v^{q_2}(y)}{|x|^{\beta_1}|x-y|^{n-\alpha}|y|^{\beta_2}}dx\\[3mm]
&=\displaystyle\int_{R^n}v^{q_2}(y)\int_{R^{n}} \frac{u^{q_1}(x)dx}{|y|^{\beta_2}|y-x|^{n-\alpha}|x|^{\beta_1}}dy
=\int_{R^n}v^{q_2}(y)v(y)dy=\int_{R^n}v^{q_2+1}(y)dy.
\end{array}
$$
Combining two results above yields (\ref{ELWHLScond2}).
\end{proof}

\subsection{Equations with Bessel potentials}

Same as (\ref{el2}), the fractional order equation
\begin{equation} \label{wuqiong}
(-\Delta)^{\alpha/2}u=u^q, \quad u>0 ~in~ R^n,
\end{equation}
is still invariant under the conformal transformation as long as
$q=\frac{n+\alpha}{n-\alpha}$. Here the fractional order
differential operator $(-\Delta)^{\alpha/2}$ can be defined via
the properties of the Riesz potential (cf. \cite{Stein}).
According to the results in \cite{CL-cpaa}
and \cite{CLO}, it is equivalent to the integral equation (\ref{hls}).
In addition, the fact
$
\delta^{-\alpha/2}=c_\alpha\int_0^\infty exp(-t\delta)
t^{\alpha/2} \frac{dt}{t}
$
shows that the kernel of the Riesz potential can be written as a
static heat kernel. Namely, besides (\ref{hls}), we obtain another
integral equation which is equivalent to (\ref{wuqiong}):
\begin{equation} \label{heat}
u(x)=\int_{R^n}u^q(y)\int_0^\infty (4\pi
t)^{\frac{\alpha-n}{2}}\exp(-\frac{|x-y|^2}{4t})\frac{dt}{t}dy.
\end{equation}

If replacing the static heat kernel
$
H(x)=c_\alpha\int_0^\infty (4\pi
t)^{\frac{\alpha-n}{2}}\exp(-\frac{|x-y|^2}{4t})\frac{dt}{t}
$
by the Bessel kernel
$
g_\alpha(x)=c_\alpha\int_0^\infty (4\pi t)^{\frac{\alpha-n}{2}}
\exp(-\frac{|x-y|^2}{4t}-\frac{t}{4\pi})\frac{dt}{t},
$
then we have a new integral equation
\begin{equation} \label{hanzaifen}
u(x)=\int_{R^n} g_\alpha(x-y)u^q(y)dy, \quad u>0 ~in~ R^n,
\end{equation}
which is equivalent to the fractional order equation (cf.
\cite{HL})
\begin{equation} \label{tongt}
(id-\Delta)^{\alpha/2}u=u^q, \quad u>0 ~in~ R^n.
\end{equation}
This equation is not invariant under the scaling (\ref{scal}).

When $\alpha=2$, (\ref{tongt}) becomes a semilinear equation
\begin{equation} \label{schr}
-\Delta u+u=u^q, \quad u>0~in~R^n.
\end{equation}
Here $q>1$. It can be used to describe the solitary wave of the
Schr\"odinger equation. A known result implied in Chapter 8 of
\cite{Caz} is $q < \frac{n+2}{n-2}$ (namely $q$ is subcritical) if
$u \in H^1(R^n)$ is a weak solution of (\ref{schr}).

Next, we investigate the relation between weak solutions and
finite energy solutions of (\ref{schr}).

\begin{proposition} \label{prop3.1}
Assume $u$ is a positive solution of (\ref{schr}). Then $u
\in C^2(R^n) \cap L^{q+1}(R^n)$ if and only if $u \in H^1(R^n)$.
In addition, $\|u\|_{H^1(R^n)}^2=\|u\|_{L^{q+1}(R^n)}^{q+1}$.
\end{proposition}

\begin{proof}
{\it Step 1.}
If $u \in H^1(R^n)$ is a weak solution, then $u \in C^2(R^n)$ (cf. \cite{Caz}).
Testing by $u\zeta_R^2$ yields
\begin{equation} \label{xiong}
\int_{B_{2R}(0)}\nabla u\nabla(u\zeta_R^2)dx+\int_{B_{2R}(0)}u^2\zeta_R^2dx
=\int_{B_{2R}(0)}u^{q+1}\zeta_R^2dx.
\end{equation}
Therefore, by the H\"older inequality, from $u \in H^1(R^n)$ we deduce that
$
\int_{B_{2R}(0)}u^{q+1}\zeta_R^2dx \leq C,
$
where $C>0$ is independent of $R$. Letting $R \to \infty$, we get $u \in L^{q+1}(R^n)$.

{\it Step 2.}
On the contrary, if $u \in C^2(R^n) \cap L^{q+1}(R^n)$, multiplying
by $u\zeta_R^2$ and integrating on $B_{2R}(0)$, we also have (\ref{xiong}).
In view of $u \in L^{q+1}(R^n)$, it follows
$$
\int_{B_{2R}(0)}u^2|\nabla \zeta_R|^2dx
\leq \|u\|_{L^{q+1}(R^n)}^2(\int_{B_{2R}(0)}|\nabla \zeta_R|^{\frac{2(q+1)}{q-1}}dx)
^{\frac{q-1}{q+1}} \leq CR^{\frac{n(q-1)}{q+1}-2}.
$$
Since $q$ is subcritical, $\lim_{R \to \infty}\int_{B_{2R}(0)}u^2|\nabla \zeta_R|^2dx=0$.
Thus, we can easily see $u \in H^1(R^n)$.

{\it Step 3.} We claim $\|u\|_{H^1(R^n)}^2=\|u\|_{L^{q+1}(R^n)}^{q+1}$.
In fact, under each assumption, $u \in C^2(R^n)$. Multiplying (\ref{schr}) by $u$,
we get
\begin{equation} \label{jubu}
\int_{B_R(0)}(|\nabla u|^2+|u|^2)dx=\int_{B_R(0)}u^{q+1}dx+\int_{\partial B_R(0)}
u\partial_{\nu}uds.
\end{equation}
By virtue of $u \in H^1(R^n) \cap L^{q+1}(R^n)$, we can find $R=R_j \to \infty$
such that
$$
R\int_{\partial B_R(0)}(|\nabla u|^2+u^{q+1})ds \to 0.
$$
Therefore, by the H\"older inequality, we have
$$\begin{array}{ll}
|\displaystyle\int_{\partial B_R(0)}u \partial_{\nu}uds|
&\leq (R\displaystyle\int_{\partial B_R(0)}
|\nabla u|^2dx)^{\frac{1}{2}}
(R\int_{\partial B_R(0)}u^{q+1}dx)^{\frac{1}{q+1}}
|\partial B_R|^{\frac{1}{2}-\frac{1}{q+1}}\\[3mm]
&\leq CR^{(n-1)(\frac{1}{2}-\frac{1}{q+1})-(\frac{1}{2}+\frac{1}{q+1})}.
\end{array}
$$
Since $q$ is subcritical,
$\lim_{R \to \infty} |\int_{\partial B_R(0)}u \partial_{\nu}uds|=0$.
Inserting this into (\ref{jubu}), we obtain
$$
\int_{R^n}(|\nabla u|^2+|u|^2)dx=\int_{R^n}u^{q+1}dx.
$$
Theorem \ref{prop3.1} is proved.
\end{proof}

\begin{proposition} \label{prop3.2}
Assume $u$ solves (\ref{hanzaifen}), then
$u \in H^{\alpha/2}(R^n)$ if and only if $u \in L^{q+1}(R^n)$. In addition,
$\|u\|_{H^{\alpha/2}(R^n)}^2=\|u\|_{L^{q+1}(R^n)}^{q+1}$.
\end{proposition}

\begin{proof}
From (\ref{hanzaifen}), we have
$\hat{u}(\xi)=\hat{g}_\alpha(\xi) (u^q)^{\wedge}(\xi)$, or $(1+4\pi^2
|\xi|^2)^{\frac{\alpha}{2}}\hat{u}(\xi)=(u^q)^{\wedge}(\xi)$.
Multiplying by $\bar{\hat{u}}$ and using the Parseval identity,
we get
$$
\int_{R^n}(1+4\pi^2|\xi|^2)^{\frac{\alpha}{2}}|\hat{u}|^2d\xi
=\int_{R^n}(u^q)^{\wedge}\bar{\hat{u}}d\xi
=\int_{R^n}u^{q+1}dx.
$$
Therefore, the proof is easy to complete.
\end{proof}

\begin{theorem} \label{th3.1}

(1) If $q<\frac{n+\alpha}{n-\alpha}$, (\ref{hanzaifen}) has a positive solution
in $L^{q+1}(R^n)$. Moreover, if $\alpha>1$, then $u \in C^1(R^n)$.

(2) If (\ref{hanzaifen}) has a positive solution $u \in C^1(R^n) \cap L^{q+1}(R^n)$,
then $q<\frac{n+\alpha}{n-\alpha}$.
\end{theorem}

\begin{proof}
(1) By the analogous argument of the existence of ground state in \cite{FQT},
we can find a critical point of
$$
E(u)=\frac{1}{2}\int_{R^n}(1+4\pi^2|\xi|^2)^{\alpha/2}|\hat{u}(\xi)|^2d\xi
-\int_{R^n}\frac{u^{q+1}(x)}{q+1}dx
$$
on the Nihari manifold $\{u \in H^{\alpha/2}(R^n)\setminus \{0\};E'(u)u=0\}$.
According to Proposition \ref{prop3.2}, $u \in L^{q+1}(R^n)$. This implies
the existence of the weak solution of (\ref{tongt}).
According to the equivalence, (\ref{hanzaifen}) also
has a finite energy solution. In addition, $u$ is radially
symmetric and decreasing about some point in $R^n$ (cf. \cite{MCh}). In the
same way to lift regularity process in \cite{Lei-MZ}, we can also deduce
the regularity of the solution from
$H^{\alpha/2}(R^n)$ to $C^1(R^n)$ by virtue of $\alpha>1$.

(2) Clearly,
$$
u(\mu x)=\mu^\alpha\int_{R^n}u^q(\mu y)\int_0^\infty (4\pi
t)^{\frac{\alpha-n}{2}}
\exp(-\frac{|x-y|^2}{4t}-\frac{\mu^2t}{4\pi})\frac{dt}{t}dy.
$$
Thus,
$$\begin{array}{ll}
x\cdot \nabla u(x)=&[\displaystyle\frac{du(\mu x)}{d\mu}]_{\mu=1}
=\alpha u
+\int_{R^n}y \cdot \nabla u^q(y) g_\alpha(x-y)dy\\[3mm]
&-\displaystyle\int_{R^n}u^q(y)\int_0^\infty (4\pi
t)^{\frac{\alpha-n}{2}}
\exp(-\frac{|x-y|^2}{4t}-\frac{t}{4\pi})\frac{t}{2\pi}\frac{dt}{t}dy.
\end{array}
$$
Multiplying by $u^q(x)$ and integrating, we get
$$\begin{array}{ll}
&\quad \displaystyle\frac{1}{q+1}\int_{R^n}x \cdot \nabla
u^{q+1}(x)dx\\[3mm]
&=\alpha \displaystyle\int_{R^n}u^{q+1}(x)dx
+\frac{q}{q+1}\int_{R^n}y\cdot\nabla u^{q+1}(y)dy\\[3mm]
&-\displaystyle\int_{R^n}\int_{R^n}u^q(x)u^q(y)\int_0^\infty (4\pi
t)^{\frac{\alpha-n}{2}}
\exp(-\frac{|x-y|^2}{4t}-\frac{t}{4\pi})\frac{t}{2\pi}\frac{dt}{t}dxdy.
\end{array}
$$
If $u \in L^{q+1}(R^n)$, we can find $R=R_j \to \infty$ such that
$R\int_{\partial B_R(0)}u^{q+1}ds \to 0$. Thus, the result above
leads to
$$\begin{array}{ll}
&\quad (\displaystyle\frac{q-1}{q+1}n-\alpha)\int_{R^n}u^{q+1}dx\\[3mm]
&=-\displaystyle\int_{R^n}\int_{R^n}u^q(x)u^q(y)\int_0^\infty
(4\pi t)^{\frac{\alpha-n}{2}} \exp(-\frac{|x-y|^2}{4t}
-\frac{t}{4\pi})\frac{t}{2\pi}\frac{dt}{t}dxdy.
\end{array}
$$
Since the right hand side is positive, we can deduce that
$\frac{q-1}{q+1}n-\alpha>0$, which implies
$q<\frac{n+\alpha}{n-\alpha}$.
\end{proof}

\vskip 3mm

Consider the system
\begin{equation} \label{bssy}
 \left \{
   \begin{array}{l}
      (id-\Delta)^{\alpha/2}u=v^{q_2}, \quad u>0~in~R^n,\\
      (id-\Delta)^{\alpha/2}v=u^{q_1}, \quad v>0~in~R^n,
   \end{array}
   \right.
\end{equation}
where $n \geq 3$, $\alpha \in (0,n)$, $q_1,q_2>0$.

According to the definition of weak solutions $u,v$ of (\ref{bssy}) in
$H^{\alpha/2}(R^n)$, for all $\phi \in H^{\alpha/2}(R^n)$, there hold
$$
Re\int_{R^n}(1+4\pi^2|\xi|^2)^{\alpha/2}\hat{u}(\xi)\bar{\hat{\phi}}(\xi)d\xi
=\int_{R^n}v^{q_2}(x)\phi(x)dx,
$$
$$
Re\int_{R^n}(1+4\pi^2|\xi|^2)^{\alpha/2}\hat{v}(\xi)\bar{\hat{\phi}}(\xi)d\xi
=\int_{R^n}u^{q_1}(x)\phi(x)dx.
$$
Here $\hat{u}$ is the Fourier transformation of $u$.

\begin{theorem} \label{th3.2}
If (\ref{bssy}) has weak positive solutions in $H^{\alpha/2}(R^n)$.
Then
\begin{equation}\label{subcri}
\frac{1}{q_1+1}+\frac{1}{q_2+1}>\frac{n-\alpha}{n}.
\end{equation}
\end{theorem}

\begin{proof}
Testing (\ref{bssy}) by $u$ and $v$ respectively, we have
$$
Re\int_{R^n}(1+4\pi^2|\xi|^2)^{\alpha/2}\hat{u}(\xi)\bar{\hat{v}}(\xi)d\xi
=\int_{R^n}v^{q_2+1}(x)dx,
$$
$$
Re\int_{R^n}(1+4\pi^2|\xi|^2)^{\alpha/2}\hat{v}(\xi)\bar{\hat{u}}(\xi)d\xi
=\int_{R^n}u^{q_1+1}(x)dx.
$$
Since the left hand sides of two equalities above are equal, it follows
\begin{equation} \label{yangjun}
Re\int_{R^n}(1+4\pi^2|\xi|^2)^{\alpha/2}\hat{u}(\xi)\bar{\hat{v}}(\xi)d\xi
=\int_{R^n}v^{q_2+1}(x)dx=\int_{R^n}u^{q_1+1}(x)dx.
\end{equation}

On the other hand, the positive weak solutions $u,v$
are the critical points of the functional
$$
E(u,v)=Re\int_{R^n}(1+4\pi^2|\xi|^2)^{\alpha/2}
\hat{u}(\xi)\bar{\hat{v}}(\xi)d\xi
-\int_{R^n}(\frac{u^{q_1+1}}{q_1+1}+\frac{v^{q_2+1}}{q_2+1})dx.
$$
Thus, the Pohozaev identity
$
[\frac{d}{d\mu}E(u(\frac{x}{\mu}),v(\frac{x}{\mu}))]_{\mu=1}=0
$
holds. By virtue of
$$
E(u(\displaystyle\frac{x}{\mu}),v(\frac{x}{\mu}))=
\mu^{n-\alpha}\displaystyle Re\int_{R^n}(1+4\pi^2|\zeta|^2)^{\alpha/2}
\hat{u}(\zeta)\bar{\hat{v}}(\zeta)d\zeta
-\mu^n\displaystyle\int_{R^n}(\frac{u^{q_1+1}}{q_1+1}
+\frac{v^{q_2+1}}{q_2+1})dy,
$$
the Pohozaev identity leads to
$$\begin{array}{ll}
(n-\alpha)\displaystyle Re\int_{R^n}(1+4\pi^2|\zeta|^2)^{\alpha/2}
\hat{u}(\zeta)\bar{\hat{v}}(\zeta)d\zeta
&+\alpha\displaystyle Re\int_{R^n}(1+4\pi^2|\zeta|^2)^{(\alpha-2)/2}
\hat{u}(\zeta)\bar{\hat{v}}(\zeta)d\zeta\\[3mm]
&=n\displaystyle\int_{R^n}(\frac{u^{q_1+1}}{q_1+1}
+\frac{v^{q_2+1}}{q_2+1})dy.
\end{array}
$$
Combining with (\ref{yangjun}), we get
\begin{equation} \label{wuyaling}
\alpha\displaystyle Re\int_{R^n}(1+4\pi^2|\zeta|^2)^{\frac{\alpha-2}{2}}
\hat{u}(\zeta)\bar{\hat{v}}(\zeta)d\zeta
=[n(\displaystyle\frac{1}{q_1+1}+\frac{1}{q_2+1})-(n-\alpha)]
\int_{R^n}u^{q_1+1}(x)dx.
\end{equation}

We claim that the left hand side of (\ref{wuyaling}) is positive.
In fact, set
$$
w(x)=\int_{R^n} g_2(x-y)v(y)dy.
$$
Then, $w>0$ belongs to $H^{\alpha/2}(R^n)$, and
$
\hat{w}=(1+4\pi^2|\xi|^2)^{-1}\hat{v}.
$
Testing (\ref{bssy}) by $w$ yields
$$
Re\int_{R^n}(1+4\pi^2|\xi|^2)^{\alpha/2}\hat{u}\bar{\hat{w}}d\xi
=\int_{R^n}v^{q_2}wdx,
$$
which implies
$$
Re\int_{R^n}(1+4\pi^2|\xi|^2)^{(\alpha-2)/2}\hat{u}\bar{\hat{v}}d\xi>0.
$$
Combining this result with (\ref{wuyaling}), we see the subcritical condition
(\ref{subcri}). Theorem \ref{th3.2} is proved.
\end{proof}

\paragraph{Remark 2.2.} If we prove the second conclusion of Theorem \ref{th3.1} by the same
way of Theorem \ref{th3.2}, the assumption of $u \in C^1(R^n)$ can be removed.

\subsection{Representation of minimum in critical case}

Consider the minimum of the following energy functional in
$H^{\alpha/2}(R^n)\setminus \{0\}$
$$
E(u)=\frac{1}{2}\int_{R^n}(1+4\pi^2|\xi|^2)^{\alpha/2}|\hat{u}(\xi)|^2d\xi
-\frac{1}{\alpha^*}\int_{R^n}u^{\alpha^*}(x)dx,
$$
where $\alpha^*=\frac{2n}{n-\alpha}$. Clearly, $\alpha^*-1$ is the critical exponent.

By the argument in \S 2.3, we know that $E(u)$ has no minimizer in
$H^{\alpha/2}(R^n)\setminus \{0\}$ in the critical case. However,
the radial function
$$
U_*(x)=a(\frac{b}{b^2+|x-x_0|^2})^{(n-\alpha)/2}, \quad
a,b>0 \quad and \quad x_0 \in R^n
$$
is the extremal the Hardy-Littlewood-Sobolev inequality (cf. \cite{L}). Furthermore,
according to the classification results in \cite{CLO} and \cite{YLi},
the radial function $U_*$ is the unique solution of (\ref{hls}). In addition, it is also
the extremal function in $\mathcal{D}^{\alpha/2,2}(R^n)\setminus \{0\}$
of the functional
$$
E_*(u)=\left[\int_{R^n}|(-\Delta)^{\alpha/4}u|^2dx\right]
\left[\int_{R^n}|u|^{2n/(n-\alpha)}dx\right]^{(\alpha-n)/n}.
$$
The classification of the solutions also provides the sharp constant in the inequality
of the critical Sobolev imbedding from $\mathcal{D}^{\alpha/2,2}(R^n)$ to
$L^{2n/(n-\alpha)}(R^n)$:
$$
c(\int_{R^n}|u|^{2n/(n-\alpha)}dx)^{(n-\alpha)/n}
\leq \int_{R^n}|(-\Delta)^{\alpha/4}u|^2dx.
$$

The following result shows the relation between the energy functionals
involving the Riesz potential and the Bessel potential in the critical case.

\begin{theorem} \label{th3.3}
$\inf\{E(u); u \in H^{\alpha/2}(R^n)\setminus
\{0\}\}=\frac{\alpha}{2n}[E_*(U_*)]^{n/\alpha}$.
\end{theorem}

\begin{proof}
The ideas in \cite{Caz} and \cite{IMN} are used here.

Write the scaling function
$$
u_{t,s}^\lambda(x)=e^{t\lambda}u(e^{-s\lambda}x),
$$
where $\lambda \geq 0$, $t\geq 0$, $t^2+s^2>0$, $\mu:=2t+(n-2)s
\geq 0$, and $\nu:=2t+ns \geq 0$. Set $\bar{\mu}=\max\{\mu,\nu\}$.

By a simply calculation, we have
$$\begin{array}{ll}
&K(u):=\displaystyle\frac{dE(u_{t,s}^\lambda)}
{d\lambda}|_{\lambda=0}\\[3mm]
&=\displaystyle\frac{\mu}{2}\int_{R^n}
(1+4\pi^2|\xi|^2)^{\frac{\alpha}{2}}|\hat{u}(\xi)|^2d\xi
+\displaystyle\frac{s\alpha}{2}\int_{R^n}
(1+4\pi^2|\xi|^2)^{\frac{\alpha-2}{2}}|\hat{u}(\xi)|^2d\xi
-\frac{\mu}{2}\int_{R^n}u^{\alpha^*}(x)dx\\[3mm]
&=\displaystyle\frac{\nu}{2}\int_{R^n}
(1+4\pi^2|\xi|^2)^{\frac{\alpha}{2}}|\hat{u}(\xi)|^2d\xi
-\frac{s\alpha}{2}\int_{R^n}(1+4\pi^2|\xi|^2)^{\frac{\alpha-2}{2}}
4\pi^2|\xi|^2|\hat{u}(\xi)|^2d\xi\\[3mm]
&\quad -\displaystyle\frac{\mu}{2}\int_{R^n}u^{\alpha^*}(x)dx.
\end{array}
$$
Similarly, if we set
$$
E_0(u)=\frac{1}{2}\int_{R^n}(2\pi
|\xi|)^{\alpha}|\hat{u}(\xi)|^2d\xi
-\frac{1}{\alpha^*}\int_{R^n}u^{\alpha^*}(x)dx,
$$
then
$$
K_0(u):=\frac{dE_0(u_{t,s}^\lambda)}{d\lambda}|_{\lambda=0}
=\frac{\mu}{2}\int_{R^n}(2\pi|\xi|)^{\alpha}|\hat{u}(\xi)|^2d\xi
-\frac{\mu}{2}\int_{R^n}u^{\alpha^*}(x)dx.
$$
Write
$$\begin{array}{ll}
&L(u):=E(u)-\frac{K(u)}{\bar{\mu}}\\[3mm]
&\quad=\left \{
   \begin{array}{lll}
   &\displaystyle\frac{s\alpha}{2}\int_{R^n}
   (1+4\pi^2|\xi|^2)^{\frac{\alpha-2}{2}}
   4\pi^2|\xi|^2|\hat{u}(\xi)|^2d\xi+(\frac{\mu}{2\nu}-
   \frac{1}{\alpha^*})\int_{R^n}u^{\alpha^*}dx, \quad
   &\mu<\nu;\\[3mm]
   &-\displaystyle\frac{s\alpha}{2\mu}\int_{R^n}
   (1+4\pi^2|\xi|^2)^{\frac{\alpha-2}{2}}|\hat{u}(\xi)|^2d\xi
   +(\frac{1}{2}-\frac{1}{\alpha^*})\int_{R^n}u^{\alpha^*}dx, \quad
   &\mu \geq \nu,
   \end{array}
   \right.
\end{array}
$$
and
$$\begin{array}{ll}
& L_0(u):=E_0(u)-\frac{K_0(u)}{\bar{\mu}}\\[3mm]
&\quad=\left \{
   \begin{array}{lll}
   &\displaystyle(\frac{1}{2}-\frac{\mu}{2\nu})
   \int_{R^n}(2\pi|\xi|)^{\alpha}
   |\hat{u}(\xi)|^2d\xi+(\frac{\mu}{2\nu}-\frac{1}{\alpha^*})
   \int_{R^n}u^{\alpha^*}dx, \quad &\mu<\nu;\\[3mm]
   &\displaystyle(\frac{1}{2}-\frac{1}{\alpha^*})
   \int_{R^n}u^{\alpha^*}dx, \quad
   &\mu \geq \nu.
   \end{array}
   \right. \end{array}
$$
When $\mu<\nu$, $s>0$ and $\frac{\mu}{2\nu}>\frac{1}{\alpha^*}$;
when $\mu \geq \nu$, $s \leq 0$. Thus, $L(u), L_0(u) \geq 0$.

In view of the parameter independence (cf. \cite{IMN}), we can
define
$$\begin{array}{ll}
&m=\inf\{E(u);K(u)=0, u \in H^{\alpha/2}(R^n) \setminus \{0\}\};\\[3mm]
&\bar{m}=\inf\{L(u);K(u) \leq 0, u \in H^{\alpha/2}(R^n) \setminus \{0\}\};\\[3mm]
&m_0=\inf\{E_0(u);K_0(u)=0, u \in \mathcal{D}^{\alpha/2,2}(R^n)
\setminus \{0\}\};\\[3mm]
&\bar{m}_0=\inf\{L_0(u);K_0(u) <0, u \in
\mathcal{D}^{\alpha/2,2}(R^n) \setminus \{0\}\}.
\end{array}
$$
Clearly, $m=\bar{m}$. Set
$$\begin{array}{ll}
&F=\{u \in \mathcal{D}^{\alpha/2,2}(R^n) \setminus \{0\};
K_0(u)<0\},\\[3mm]
&\tilde{F}=\{u \in \mathcal{D}^{\alpha/2,2}(R^n) \setminus \{0\};
K_0(u) \leq 0\},\\[3mm]
&\bar{F}=\{u \in \mathcal{D}^{\alpha/2,2}(R^n) \setminus \{0\};
K_0(u)=0\}.
\end{array}
$$

We claim $F=\cup_{\lambda>0} \{u \in \mathcal{D}^{\alpha/2,2}(R^n) \setminus \{0\};
K_0(u_{t,s}^\lambda)=0\}$. Once it holds, then
$m_0=\bar{m}_0$.

In fact, for any $\lambda>0$, if $K_0(u_{t',s'}^\lambda)=0$, then
$$
e^{[2t'+(n-\alpha)s']\lambda}\int_{R^n}(2\pi|\xi|)^{\alpha}
|\hat{u}(\xi)|^2d\xi=e^{[2t'+(n-\alpha)s']\frac{n\lambda}{n-\alpha}}
\int_{R^n}u^{\alpha^*}dx.
$$
This leads to $K_0(u)<0$, and hence
$F\supset\cup_{\lambda>0} \{u \in \mathcal{D}^{\alpha/2,2}(R^n) \setminus \{0\};
K_0(u_{t,s}^\lambda)=0\}$.

On the other hand, for any $u \in F$, there holds $\int_{R^n}(2\pi|\xi|)^{\alpha}
|\hat{u}(\xi)|^2d\xi<\int_{R^n}u^{\alpha^*}dx$. Thus, we can find $\lambda_*>0$
such that
$$
e^{[2t'+(n-\alpha)s']\lambda_*}\int_{R^n}(2\pi|\xi|)^{\alpha}
|\hat{u}(\xi)|^2d\xi=e^{[2t'+(n-\alpha)s']\frac{n\lambda_*}{n-\alpha}}
\int_{R^n}u^{\alpha^*}dx.
$$
This shows $u \in \{u \in \mathcal{D}^{\alpha/2,2}(R^n) \setminus \{0\};
K_0(u_{t,s}^{\lambda_*})=0\}$.

In addition, it is easy to see that $F$ is dense in $\tilde{F}$, which implies
\begin{equation} \label{baog}
\bar{m}_0=\inf\{L_0(u);K_0(u) \leq 0, u \in
\mathcal{D}^{\alpha/2,2}(R^n) \setminus \{0\}\}.
\end{equation}

Set $G=\{u \in H^{\alpha/2}(R^n) \setminus \{0\}; K(u) \leq 0\}$.
Clearly, $G \subset \tilde{F}$.

Noting
$$\begin{array}{ll}
K(u_{t',s'}^\lambda)=&\displaystyle\frac{\mu}{2}e^{\lambda(2t'+(n-\alpha)s')}
\int_{R^n}(e^{2s\lambda}+4\pi^2|\xi|^2)^{\frac{\alpha}{2}}|\hat{u}(\xi)|^2d\xi\\[3mm]
&+\displaystyle\frac{s\alpha}{2}e^{\lambda(2t'+(n-\alpha+2)s')}\int_{R^n}
(e^{2s\lambda}+4\pi^2|\xi|^2)^{\frac{\alpha-2}{2}}|\hat{u}(\xi)|^2d\xi\\[3mm]
&-\displaystyle\frac{\mu}{2}e^{\lambda(\alpha^* t'+ns')}\int_{R^n}u^{\alpha^*}(x)dx.
\end{array}
$$
we can deduce by taking $t'=\frac{n-\alpha}{2}$ and $s'=-1$ that
\begin{equation} \label{kmmo}
\lim_{\lambda \to +\infty}K(u_{(n-\alpha)/2,-1}^\lambda)=K_0(u),
\end{equation}
Similarly, we also get
\begin{equation} \label{lmmo}
\lim_{\lambda \to +\infty}L(u_{(n-\alpha)/2,-1}^\lambda)=L_0(u),
\end{equation}
Clearly, (\ref{kmmo}) shows that $G$ is dense in $\tilde{F}$. Combining with (\ref{baog})
yields
$$
\bar{m}_0=\inf\{L_0(u);K(u) \leq 0, u \in
\mathcal{D}^{\alpha/2,2}(R^n) \setminus \{0\}\}.
$$
In addition, (\ref{lmmo}) implies
$$
\inf\{L_0(u);K(u) \leq 0, u \in \mathcal{D}^{\alpha/2,2}(R^n)
\setminus \{0\}\}=\bar{m}.
$$
Therefore, $\bar{m}_0=\bar{m}$.

The argument above shows that $m=\bar{m}=m_0=\bar{m}_0$.

Take $t=0$, then $\alpha^* \mu=2\nu$. Thus,
$$\begin{array}{ll}
&\quad m=m_0\\[3mm]
&=\inf\left\{\displaystyle\frac{\alpha}{2n}\int_{R^n}(2\pi|\xi|)^\alpha
|\hat{u}(\xi)|^2d\xi;\int_{R^n}(2\pi|\xi|)^\alpha
|\hat{u}(\xi)|^2d\xi=\int_{R^n}u^{\alpha^*}(x)dx\right\}\\[3mm]
&=\inf\left\{\displaystyle\frac{\alpha}{2n}\int_{R^n}(2\pi|\xi|)^\alpha
|\hat{u}(\xi)|^2d\xi \left[\frac{\int_{R^n}(2\pi|\xi|)^\alpha
|\hat{u}(\xi)|^2d\xi}{\int_{R^n}u^{\alpha^*}(x)dx}\right]^{\frac{n-\alpha}{\alpha}};u
\in \mathcal{D}^{\alpha/2,2}(R^n) \setminus \{0\}\right\}\\[3mm]
&=\displaystyle\frac{\alpha}{2n}\inf\left\{\left[\frac{\int_{R^n}(2\pi|\xi|)^\alpha
|\hat{u}(\xi)|^2d\xi}{(\int_{R^n}
u^{\alpha^*}(x)dx)^{(n-\alpha)/n}}\right]^{\frac{n}{\alpha}};u \in
\mathcal{D}^{\alpha/2,2}(R^n) \setminus \{0\}\right\}\\[3mm]
&=\displaystyle\frac{\alpha}{2n}c_*^{n/\alpha}.
\end{array}
$$
Here $c_*$ is the sharp constant of the inequality
$$
c(\int_{R^n}u^{\alpha^*}(x)dx)^{(n-\alpha)/n} \leq
\int_{R^n}|(-\Delta)^{\alpha/4}u(x)|^2dx.
$$
According to the classification result in \cite{CLO}, we know that
the corresponding minimizer in $\mathcal{D}^{\alpha/2,2}(R^n)
\setminus \{0\}$ is $U_*$.
\end{proof}

\section{Caffarelli-Kohn-Nirenberg type equations}

Consider the Caffarelli-Kohn-Nirenberg inequality
$$
(\int_{R^n}\frac{|u|^{q+1}}{|x|^{b(q+1)}}dx)^{p/(q+1)}
\leq C_{a,b}\int_{R^n}\frac{|\nabla u|^p}{|x|^{ap}}dx,
$$
where $n \geq 3$, $p>1$, $0 \leq a<\frac{n-p}{p}$, and
$a \leq b \leq a+1$. Since the scaling function $u_\mu(x)$ also satisfies this inequality,
by a simple calculation we can see $q=\frac{np}{n-p+p(b-a)}-1$.

The extremal functions in $\mathcal{D}_a^{1,p}(R^n) \setminus \{0\}$
can be obtained by investigating the functional
$$
E(u)=\int_{R^n}\frac{|\nabla u|^p}{|x|^{ap}}dx
(\int_{R^n}\frac{|u|^{q+1}}{|x|^{b(q+1)}}dx)^{-p/(q+1)}.
$$
Here $\mathcal{D}_a^{1,p}(R^n)$ is the completion of $C_0^\infty(R^n)$
with respect to the norm $\||x|^{-a}\nabla u\|_{L^p(R^n)}$.
Clearly, the extremal function satisfies the Pohozaev identity
$
\frac{d}{d\mu}E(u(\frac{x}{\mu}))|_{\mu=1}=0.
$
Noting
$
E(u(\frac{x}{\mu}))=\mu^{n-p(a+1)-\frac{pn}{q+1}+pb}E(u(x)),
$
we also obtain $q=\frac{np}{n-p+p(b-a)}-1$.

Consider the Euler-Lagrange equation which the extremal function of $E(u)$ satisfies:
\begin{equation} \label{CKN}
-div(\frac{1}{|x|^{ap}}|\nabla u|^{p-2}\nabla u)
=\frac{1}{|x|^{b(q+1)}}u^q, \quad u>0~in ~R^n.
\end{equation}

By a direct calculation we also deduce that (\ref{CKN}) and the energy
$\int_{R^n}\frac{|u|^{q+1}}{|x|^{b(q+1)}}dx$ are invariant under
the scaling (\ref{scal}) if and only if $q=\frac{np}{n-p+p(b-a)}-1$.
If $a=b=0$, this exponent $q=\frac{np}{n-p}-1$
is the critical condition for the existence (cf. \cite{SZ}).

We consider the relation between the finite energy solutions and
the critical exponents.

\begin{theorem} \label{th4.1}
(1) If $u \in \mathcal{D}_a^{1,p}(R^n)$ is a weak solution of
(\ref{CKN}), then $|x|^{-b}u \in L^{q+1}(R^n)$. Moreover, if $u
\in C^2(R^n)$, then
\begin{equation} \label{eneq}
\||x|^{-a}\nabla u\|_{L^p(R^n)}^p
=\||x|^{-b}u\|_{L^{q+1}(R^n)}^{q+1}.
\end{equation}

(2) On the contrary, assume $u \in C^2(R^n)$ solves (\ref{CKN}),
and $|x|^{-b}u \in L^{q+1}(R^n)$. If
$$
\int_{R^n}(\frac{u}{|x|^b})^{\frac{np}{n+p(b-a-1)}}dx<\infty.
$$
then $u \in \mathcal{D}_a^{1,p}(R^n)$ and (\ref{eneq}) still
holds.
\end{theorem}

\begin{proof}
{\it Step 1.} If $u \in \mathcal{D}_a^{1,p}(R^n)$ is a weak
solution of (\ref{CKN}), testing by $u\zeta_R^p$ yields
\begin{equation} \label{songkui}
\int_{R^n}|x|^{-ap}|\nabla u|^{p-2} \nabla u \nabla(u\zeta_R^p)dx
=\int_{R^n}|x|^{-b(q+1)}u^{q+1}\zeta_R^pdx.
\end{equation}
By the Young inequality, it follows
$$
\int_{R^n}|x|^{-b(q+1)}u^{q+1}\zeta_R^pdx \leq
C\int_{R^n}|x|^{-ap}|\nabla u|^p\zeta_R^pdx +
C\int_{R^n}|x|^{-ap}u^p|\nabla \zeta_R|^pdx.
$$
Using the H\"older inequality and the Caffarelli-Kohn-Nirenberg
inequality, we get
$$\begin{array}{ll}
\displaystyle\int_{R^n}\frac{u^p}{|x|^{ap}}|\nabla \zeta_R|^pdx
&\leq \displaystyle\frac{C}{R^p}[\int_{R^n}(\frac{u}{|x|^b})^{\frac{np}{n+p(b-a-1)}}dx]^{1-\frac{p(1+a-b)}{n}}
(\int_{B_{2R}(0)}|x|^{\frac{(b-a)n}{a+1-b}}dx)^{\frac{p(1+a-b)}{n}}\\[3mm]
&\leq C\displaystyle[\int_{R^n}(\frac{u}{|x|^b})^{\frac{np}{n+p(b-a-1)}}dx]^{1-\frac{p(1+a-b)}{n}}
\leq C\int_{R^n}|x|^{-ap}|\nabla u|^pdx.
\end{array}
$$
Combining two results above and letting $R \to \infty$,
we obtain $|x|^{-b}u \in L^{q+1}(R^n)$.

In addition, if $u \in C^2(R^n)$ solves (\ref{CKN}), multiplying
by $u$ and integrating on $B_{2R}(0)$, we have
\begin{equation} \label{sunwu}
\int_{B_{2R}(0)}|x|^{-ap}|\nabla u|^pdx -\int_{\partial
B_{2R}(0)}|x|^{-ap}u|\nabla u|^{p-2}\partial_{\nu}uds
=\int_{B_{2R}(0)}|x|^{-b(q+1)}u^{q+1}dx.
\end{equation}
Using the H\"older inequality, we get
\begin{equation} \label{huangm}
\begin{array}{ll}
&\quad |\displaystyle\int_{\partial B_{2R}(0)}|x|^{-ap}u|\nabla u|^{p-2}\partial_{\nu}uds|\\[3mm]
&\leq (\displaystyle\int_{\partial B_{2R}(0)}|x|^{-ap}|\nabla
u|^pds)^{1-\frac{1}{p}}
[\int_{\partial B_{2R}}(\frac{u}{|x|^b})^{\frac{np}{n-p(a+1-b)}}ds]^{\frac{1}{p}-\frac{(1+a-b)}{n}}\\[3mm]
&\quad \cdot (\displaystyle\int_{\partial B_{2R}(0)}|x|^{\frac{(b-a)n}{a+1-b}}ds)^{\frac{a+1-b}{n}}\\[3mm]
&\leq C(R\displaystyle\int_{\partial B_{2R}(0)}|x|^{-ap}|\nabla
u|^pds)^{1-\frac{1}{p}}
[R\int_{\partial B_{2R}}(\frac{u}{|x|^b})^{\frac{np}{n-p(a+1-b)}}ds]^{\frac{1}{p}-\frac{(1+a-b)}{n}}\\[3mm]
&\quad \cdot
R^{[n-1+\frac{(b-a)n}{a+1-b}]\frac{a+1-b}{n}-1+\frac{1}{p}-\frac{1}{p}+\frac{a+1-b}{n}}.
\end{array}
\end{equation}
In view of $u \in \mathcal{D}_a^{1,p}(R^n)$, by the
Caffarelli-Kohn-Nirenberg inequality, there holds
$$
\int_{R^n}(\frac{u}{|x|^b})^{\frac{np}{n-p(a+1-b)}}dx<\infty,
$$
and hence we can find $R=R_j \to \infty$ such that
$$
R\int_{\partial B_{2R}(0)}|x|^{-ap}|\nabla u|^pds+ R\int_{\partial
B_{2R}(0)}(\frac{u}{|x|^b})^{\frac{np}{n-p(a+1-b)}}ds \to 0.
$$
Therefore, it follows from (\ref{huangm}) that
$$
|\int_{\partial B_{2R}(0)}|x|^{-ap}u|\nabla
u|^{p-2}\partial_{\nu}uds| \to 0
$$
as $R \to \infty$. Inserting this into (\ref{sunwu}) and letting
$R=R_j \to \infty$, we get (\ref{eneq}).

{\it Step 2.} If $u \in C^2(R^n)$, multiplying (\ref{CKN}) by
$u\zeta_R^p$ and integrating, we also obtain (\ref{songkui}).
Using the Young inequality, we get
$$
\int_{R^n}\frac{|\nabla u|^p}{|x|^{ap}}\zeta_R^pdx \leq
\int_{R^n}\frac{u^{q+1}}{|x|^{b(q+1)}}\zeta_R^pdx
+\frac{1}{2}\int_{R^n}\frac{|\nabla u|^p}{|x|^{ap}}\zeta_R^pdx
+C\int_{R^n}\frac{u^p}{|x|^{ap}}|\nabla \zeta_R|^pdx.
$$
This result, together with $|x|^{-b}u \in L^{q+1}(R^n)$, implies
\begin{equation} \label{huijia}
\int_{R^n}\frac{|\nabla u|^p}{|x|^{ap}}\zeta_R^pdx \leq
C+C\int_{R^n}\frac{u^p}{|x|^{ap}}|\nabla \zeta_R|^pdx.
\end{equation}
If $\int_{R^n}(\frac{u}{|x|^b})^{\frac{np}{n+p(b-a-1)}}dx<\infty$,
$$
\int_{R^n}\frac{u^p}{|x|^{ap}}|\nabla \zeta_R|^pdx \leq
\frac{1}{R^p}
(\int_{R^n}(\frac{u}{|x|^b})^{\frac{np}{n-p(a+1-b)}}dx)^{1-\frac{p(a+1-b)}{n}}
(\int_{B_{2R}}|x|^{\frac{n(b-a)}{a+1-b}}dx)^{\frac{p(a+1-b)}{n}}
\leq C.
$$
Thus, we can see $\int_{R^n}\frac{u^p}{|x|^{ap}}|\nabla
\zeta_R|^pdx<\infty$. Inserting this into (\ref{huijia}), we
have $|x|^{-a}\nabla u \in L^p(R^n)$, and hence $u \in \mathcal{D}_a^{1,p}(R^n)$.
Similar to Step 1, we also obtain
(\ref{eneq}). Theorem \ref{th4.1} is proved.
\end{proof}

\begin{theorem} \label{th4.2}
Eq. (\ref{CKN}) has a solution in $C^2(R^n) \cap \mathcal{D}_a^{1,p}(R^n)$
if and only if $q=\frac{np}{n-p+p(b-a)}-1$.
\end{theorem}

\begin{proof}
If $q=\frac{np}{n-p+p(b-a)}-1$, we know that the following extremal function of
the Caffarelli-Kohn-Nirenberg inequality is a finite energy solution
$$
U_{a,b}(x)=c_0(\frac{n-p-pa}{1+|x|^{\frac{p(n-p-pa)(1+a-b)}
{(p-1)[n-(1+a-b)p]}}})^{\frac{n-p(1+a-b)}{p(1+a-b)}}
$$
with $c_0=[n(p-1)^{1-p}(n-p(1+a-b))^{-1}]^{\frac{n-p(1+a-b)}{p^2(1+a-b)}}$.
We verify the sufficiency.

Next, we prove the necessity.
Multiplying (\ref{CKN}) by $(x\cdot \nabla u)$ and integrating on $B_R(0)$, we get
\begin{equation} \label{xiuli}
\int_{B_R(0)}\frac{|\nabla u|^{p-2}}{|x|^{ap}}\nabla u \nabla(x\cdot \nabla u)dx
-\int_{\partial B_{R}(0)}\frac{|\nabla u|^{p-2}}{|x|^{ap}}|\partial_{\nu}u|^2ds
=\int_{B_{R}(0)}\frac{u^{q}(x\cdot \nabla u)}{|x|^{b(q+1)}}dx.
\end{equation}
Integrating by parts, we obtain
\begin{equation} \label{xiuli1}
\begin{array}{ll}
&\displaystyle\int_{B_R(0)}\frac{|\nabla u|^{p-2}}{|x|^{ap}}\nabla u \nabla(x\cdot \nabla u)dx
=\int_{B_R(0)}|x|^{-ap}[|\nabla u|^{p}+\frac{1}{p}x\cdot \nabla(|\nabla u|^p)]dx\\[3mm]
&=\displaystyle\int_{B_R(0)}\frac{|\nabla u|^p}{|x|^{ap}}
+\frac{R}{p}\int_{\partial B_{R}(0)}\frac{|\nabla u|^{p}}{|x|^{ap}}ds
-\frac{n-ap}{p}\int_{B_R(0)}\frac{|\nabla u|^{p}}{|x|^{ap}}dx,
\end{array}
\end{equation}
and
\begin{equation} \label{xiuli2}
\begin{array}{ll}
&\displaystyle\int_{B_{R}(0)}\frac{u^{q}(x\cdot \nabla u)}{|x|^{b(q+1)}}dx
=\frac{1}{q+1}\int_{B_{R}(0)}\frac{x\cdot \nabla u^{q+1}}{|x|^{b(q+1)}}dx\\[3mm]
&=\displaystyle\frac{R}{q+1}\int_{\partial B_{R}(0)}\frac{u^{q+1}}{|x|^{b(q+1)}}ds
-\frac{n-b(q+1)}{q+1}\int_{B_R(0)}\frac{u^{q+1}}{|x|^{b(q+1)}}dx.
\end{array}
\end{equation}
According to the first conclusion of Theorem \ref{th4.1},
$\frac{\nabla u}{|x|^a} \in L^p(R^n)$ implies $\frac{u}{|x|^b} \in
L^{q+1}(R^n)$. We can find $R=R_j \to \infty$ such that
\begin{equation} \label{xiuli3}
R\int_{\partial B_{R}(0)}\frac{|\nabla u|^{p}}{|x|^{ap}}ds
+R\int_{\partial B_{R}(0)}\frac{u^{q+1}}{|x|^{b(q+1)}}ds \to 0.
\end{equation}

Inserting (\ref{xiuli1}) and (\ref{xiuli2}) into (\ref{xiuli}), and using
(\ref{xiuli3}), we have
$$
(1-\frac{n-ap}{p})\int_{R^n}\frac{|\nabla u|^{p}}{|x|^{ap}}dx
=-\frac{n-b(q+1)}{q+1}\int_{R^n}\frac{u^{q+1}}{|x|^{b(q+1)}}dx.
$$
Combining with (\ref{eneq}) yields $1-\frac{n-ap}{p} =-\frac{n-b(q+1)}{q+1}$,
which implies $q=\frac{np}{n-p+p(b-a)}-1$.
\end{proof}

Next we consider the system with weight
\begin{equation} \label{CKNsy}
 \left \{
   \begin{array}{l}
      -div(|x|^{-ap}|\nabla u|^{p-2}\nabla u)
=|x|^{-b(q_2+1)}v^{q_2}, \quad u>0~in ~R^n;\\
      -div(|x|^{-ap}|\nabla v|^{p-2}\nabla v)
=|x|^{-b(q_1+1)}u^{q_1}, \quad v>0~in ~R^n.
   \end{array}
   \right.
\end{equation}

\begin{theorem} \label{th4.7}
Under the scaling transformation (\ref{scalsy}), the system
(\ref{CKNsy}) and the energy functionals
$$
\int_{R^n}\frac{u^{q_1+1}}{|x|^{b(q_1+1)}}dx\quad  and
\int_{R^n}\frac{v^{q_2+1}}{|x|^{b(q_2+1)}}dx,
$$
are invariant if and only if one of the two degenerate
conditions holds: $p=2$ and $q_1=q_2$.

Moreover, if $p=2$, the critical condition is
\begin{equation} \label{cdd1}
\frac{1}{q_1+1}+\frac{1}{q_2+1}=\frac{n+2(b-a-1)}{n}.
\end{equation}
If $q_1=q_2$, the critical condition is $q_1=q_2=\frac{np}{n-p+p(b-a)}-1$.
In addition, (\ref{CKNsy}) is reduced to a single equation
(\ref{CKN}) in the weak sense (i.e. $u,v \in \mathcal{D}_a^{1,p}(R^n)$).
\end{theorem}

\begin{proof}
{\it Step 1.} By calculation, we have
$$
-div[|x|^{-ap}|\nabla u_\mu(x)|^{p-2}\nabla u_\mu(x)]
=-\mu^{(p-1)\sigma_1+(a+1)p-b(q_2+1)-\sigma_2q_2}
|x|^{-b(q_2+1)}v_\mu^{q_2}(x).
$$
If (\ref{CKNsy}) is invariant under the scaling (\ref{scalsy}),
there hold $(p-1)\sigma_1+(a+1)p-b(q_2+1)=\sigma_2q_2$
and $(p-1)\sigma_2+(a+1)p-b(q_1+1)-\sigma_1q_1$.
Thus,
$$
\sigma_1=\frac{p(q_2+p-1)(a+1-b)}{q_1q_2-(p-1)^2}-b,\quad
\sigma_2=\frac{p(q_1+p-1)(a+1-b)}{q_1q_2-(p-1)^2}-b.
$$
In addition, the energy functionals
$\int_{R^n}\frac{u^{q_1+1}}{|x|^{b(q_1+1)}}dx$ and
$\int_{R^n}\frac{v^{q_2+1}}{|x|^{b(q_2+1)}}dx$
are invariant implies $\sigma_1=\frac{n}{q_1+1}-b$ and
$\sigma_2=\frac{n}{q_2+1}-b$. Therefore,
\begin{equation} \label{liandeng}
\frac{n[q_1q_2-(p-1)^2]}{p(a+1-b)}
=(q_1+1)(q_2+p-1)=(q_2+1)(q_1+p-1).
\end{equation}
The latter equality implies $p=2$ or $q_1=q_2$.

On the contrary, the argument above also shows that if $p=2$ or $q_1=q_2$
holds, then system and the energy functionals are invariant under
the scaling (\ref{scalsy}).

{\it Step 2.} If $p=2$, (\ref{CKNsy}) becomes a Laplace system.
Thus, the former equality of (\ref{liandeng}) implies (\ref{cdd1}).
In particular, if $a=0$, it is identical with (1) of Remark 2.1.

If $q_1=q_2$, we denote them by $q$. The former equality of (\ref{liandeng}) implies
$q=\frac{np}{n-p+p(b-a)}-1$. In addition, by the definition of weak solutions, for any
$\phi \in C_0^\infty(R^n)$, there holds
$$
\int_{R^n}|x|^{-ap}(|\nabla u|^{p-2}\nabla u-|\nabla u|^{p-2}\nabla u)
\nabla\phi dx=\int_{R^n}|x|^{-b(q+1)}(v^q-u^q)\phi dx.
$$
Since $C_0^\infty(R^n)$ is dense in $\mathcal{D}_a^{1,p}(R^n)$,
we can take $\phi=u-v$. Noting the monotonicity inequality
$
(|\nabla u|^{p-2}\nabla u-|\nabla u|^{p-2}\nabla u)
\cdot \nabla(u-v) \geq 0,
$
we get
$$
\int_{R^n}|x|^{-b(q+1)}(u^q-v^q)(u-v) dx \leq 0.
$$
By the integral mean value theorem we have $u=v$ a.e. on $R^n$. Namely, (\ref{CKNsy})
is reduced to (\ref{CKN}).
\end{proof}

\paragraph{Remark 3.1.} Different from the cases $p=2$ and $q_1=q_2$,
(\ref{CKNsy}) has no variational structure.

\section{k-Hessian equations}

Tso \cite{Tso1} obtained the critical exponent and the existence/nonexistence
results for the k-Hessian equation on the bounded domain. Other related work can be
seen in \cite{CW} and the references therein.
Here we consider the following k-Hessian equation on $R^n$
\begin{equation} \label{hesse}
F_k(D^2u)=(-u)^q, \quad u<0~in~R^n.
\end{equation}
Here $F_k[D^2u]=S_k(\lambda(D^2u))$,
$\lambda(D^2u)=(\lambda_1,\lambda_2,\cdots,\lambda_n)$ with
$\lambda_i$ being eigenvalues of the Hessian matrix $(D^2u)$, and
$S_k(\cdot)$ is the $k$-th symmetric function:
$$
S_k(\lambda)=\sum_{1\leq i_1<\cdots<i_k \leq n}
\lambda_{i_1}\lambda_{i_2}\cdots\lambda_{i_k}.
$$
Two special cases are $F_1[D^2u]=\Delta u$ and $F_n[D^2u]=det(D^2u)$.

In this section, we assume $1<k<n/2$,

When $q$ is not larger than the Serrin exponent: $q \leq \frac{nk}{n-2k}$, (\ref{hesse})
has no negative solution (cf. \cite{Ou}, \cite{PV} and \cite{PV2}). Furthermore,
if $R(x)$ is double bounded, namely there exists $C>1$ such that
$C^{-1} \leq R(x) \leq C$, we can see the analogous result still holds.

\begin{theorem} \label{th4.11}
If $q\leq \frac{nk}{n-2k}$, then
\begin{equation} \label{hessian}
F_k(D^2u)=R(x)(-u)^q, \quad u<0~in~R^n,
\end{equation}
has no negative solution satisfying $\inf_{R^n}(-u)=0$
for any double bounded coefficient $R(x)$.
\end{theorem}

\begin{proof}
If (\ref{hessian}) has some negative solution $u$ satisfying
$\inf_{R^n}(-u)=0$ for some double bounded $R(x)$, then $-u$
solves an integral equation
$$
-u(x)=K(x)\int_0^\infty
\left[\frac{\int_{B_t(x)}(-u)^q(y)dy}{t^{n-2k}}\right]^{\frac{1}{k}}
\frac{dt}{t}, \quad u<0~in ~R^n,
$$
where $K(x)$ is also double bounded.
However, by the Wolff potential estimates (cf. \cite{LeiLi}), we
know this integral equation has no positive solution for any
double bounded fonction.
\end{proof}

Hereafter, we always assume that $q$ is larger than the Serrin
exponent $\frac{nk}{n-2k}$:
\begin{equation} \label{SERRIN}
q>\frac{nk}{n-2k}.
\end{equation}

\begin{theorem} \label{th4.12}
If (\ref{SERRIN}) holds, then (\ref{hessian}) has radial solutions
with the fast and the slow decay rates respectively
for some double bounded functions $R(x)$.
\end{theorem}

\begin{proof}
Clearly, if the following ODE has solution $U(r)$
\begin{equation} \label{ode}
C_{n-1}^{k-1}(\frac{U_r}{r})^{k-1}U_{rr}+C_{n-1}^k(\frac{U_r}{r})^k=R(r)(-U)^q,
\end{equation}
then $u(x)=U(r)$ solves (\ref{hessian}). Here $C_{n-1}^{k-1}$
and $C_{n-1}^k$ are combinatorial constants,
and $R(r)$ is a double bounded function.

We search the radial solution as the form
\begin{equation} \label{radial}
u(x)=U(r)=-(1+r^2)^{-\theta}, \quad r=|x|,
\end{equation}
where $\theta>0$ will be determined later.

By a direct calculation, we get
$$
U_r=\frac{2\theta r}{(1+r^2)^{\theta+1}}, \quad
U_{rr}=\frac{2\theta}{(1+r^2)^{\theta+1}}
\left[1-2(\theta+1)\frac{r^2}{1+r^2}\right].
$$
Thus, the left hand side of (\ref{ode})
\begin{equation} \label{yinan}
\begin{array}{ll}
&\quad \displaystyle C_{n-1}^{k-1}(\frac{U_r}{r})^{k-1}U_{rr}+C_{n-1}^k(\frac{U_r}{r})^k\\[3mm]
&=\displaystyle\left[\frac{2\theta}{(1+r^2)^{\theta+1}}\right]^k
\left[\frac{C_{n-1}^{k-1}+C_{n-1}^k
+[C_{n-1}^k-C_{n-1}^{k-1}(2\theta+1)]r^2}{1+r^2}\right].
\end{array}
\end{equation}

In view of (\ref{SERRIN}), it follows
$$
\frac{2k}{q-k}<\frac{n-2k}{k}.
$$
We next determine that the decay rate $2\theta$ is either the fast rate $\frac{n-2k}{k}$
or the slow rate $\frac{2k}{q-k}$.

In fact, if $C_{n-1}^k-C_{n-1}^{k-1}(2\theta+1)>0$, then $2\theta<\frac{n-2k}{k}$.
We choose $\theta$ such that $2\theta=\frac{2k}{q-k}$, and from (\ref{yinan}) we can
see that
$$
C_{n-1}^{k-1}(\frac{U_r}{r})^{k-1}U_{rr}+C_{n-1}^k(\frac{U_r}{r})^k
=R(r)(1+r^2)^{-(\theta+1)k}=R(r)(1+r^2)^{-\theta q}.
$$
This result shows that $U(r)$ as the form (\ref{radial}) with $2\theta=\frac{k+1}{q-k}$
solves (\ref{ode}).

In addition, if $C_{n-1}^k-C_{n-1}^{k-1}(2\theta+1)=0$, then $2\theta=\frac{n-2k}{k}$.
Let $q=\frac{(n+2)k}{n-2k}$. Thus, from (\ref{yinan})
we can also deduce
\begin{equation} \label{express}
C_{n-1}^{k-1}(U_r)^{k-1}U_{rr}+C_{n-1}^kr^{-1}(U_r)^k
=R(r)(1+r^2)^{-(\theta+1)k-1}=R(r)(1+r^2)^{-\theta q},
\end{equation}
and hence $U(r)$ as the form (\ref{radial}) with $2\theta=\frac{n-2k}{k}$
also solves (\ref{ode}).
\end{proof}

Come back the equation (\ref{hesse}). We consider the critical and the supercritical cases.

\begin{theorem} \label{th4.6}
Under the scaling transformation (\ref{scal}), the equation
(\ref{hesse}) and the energy functional $\|(-u)\|_{L^{q+1}(R^n)}$
are invariant if and only if $q=\frac{(n+2)k}{n-2k}$.
\end{theorem}

\begin{proof}
Clearly, if we denote $\mu x$ by $y$, then
$$
F_k(D^2u_\mu(x))=\mu^{-k(\sigma+2)}F_k(D^2u(y))=\mu^{-k(\sigma+2)}(-u)^q(y)
=\mu^{q\sigma-k(\sigma+2)}(-u_\mu)^q(x)
$$
implies $\sigma=\frac{2k}{q-k}$. On the other hand,
$
\int_{R^n}(-u_\mu)^{q+1}dx=\mu^{(q+1)\sigma-n}\int_{R^n}(-u)^{q+1}dx
$
implies $\sigma=\frac{n}{q+1}$. Eliminating $\sigma$ yields $q=\frac{(n+2)k}{n-2k}$.
The necessity is complete.

On the contrary, the argument above still works for the sufficiency.
\end{proof}

The following result shows that (\ref{hesse}) has no finite energy solution
when $q$ is not the critical exponent.

\begin{theorem} \label{th4.5}
Eq. (\ref{hesse}) has negative solution $u$ belonging to $C^2(R^n)\cap
L^{q+1}(R^n)$ if and only if $q=\frac{(n+2)k}{n-2k}$.
\end{theorem}

\begin{proof}
{\it Sufficiency.}
If $q=\frac{(n+2)k}{n-2k}$, Theorem \ref{th4.12} implies
(\ref{ode}) has a radial solution as the form (\ref{radial}) with
the fast decay rate $2\theta=\frac{n-2k}{k}$. In addition, (\ref{yinan})
shows that $R(r)$ is a constant
$$
R(x)=C_*:=(\frac{n-2k}{k})^k(C_{n-1}^{k-1}+C_{n-1}^k).
$$
Thus, setting
$$
L:= C_*^{\frac{1}{q-k}}, \quad and \quad
V(r):=LU(r)=-L(1+r^2)^{-\frac{n-2k}{2k}},
$$
we know that $v(x)=V(|x|)$ is a radial solution of (\ref{hesse})
in $C^2(R^n)\cap L^{q+1}(R^n)$.

{\it Necessity.}
Multiply (\ref{hesse}) by $-u\zeta_R$ and integrate on
$B_{2R}(0)$. Noting $u \in L^{q+1}(R^n)$ and then letting
$R \to \infty$, we obtain $uF_k(D^2u) \in L^1(R^n)$, and hence
\begin{equation} \label{fangyu}
-\int_{R^n}uF_k(D^2u)dx=\int_{R^n}(-u)^{q+1}dx.
\end{equation}
In addition, (\ref{hesse}) is the Euler-Lagrange equation for the functional
$$
E_k(u)=\frac{-1}{k+1}\int_{R^n}uF_k(D^2u)dx-\int_{R^n}\frac{(-u)^{q+1}}{q+1}dx.
$$
By virtue of
$$
E_k(u(\frac{x}{\mu}))=\frac{-\mu^{n-2k}}{k+1}\int_{R^n}uF_k(D^2u)dx
-\frac{\mu^n}{q+1}\int_{R^n}(-u)^{q+1}dx,
$$
the Pohozaev identity $\frac{d}{d\mu}E_k(u(\frac{x}{\mu}))|_{\mu=1}=0$
shows
$$
\frac{n-2k}{k+1}\int_{R^n}uF_k(D^2u)dx
+\frac{n}{q+1}\int_{R^n}(-u)^{q+1}dx=0.
$$
Inserting (\ref{fangyu}) into this result, we obtain
$\frac{n-2k}{k+1}=\frac{n}{q+1}$, which implies
$q=\frac{(n+2)k}{n-2k}$.
\end{proof}

Theorem \ref{th4.5} shows (\ref{hesse}) has a radial
solution with fast decay rate $\frac{n-2k}{k}$ when
$q$ is a critical exponent.

In the supercritical case,
Theorem \ref{th4.5} implies that the solution of (\ref{hesse})
is not the finite energy solution, and hence the decay rate
should be slower than $\frac{n-2k}{k}$.

\begin{theorem} \label{th4.10}
If $q>\frac{(n+2)k}{n-2k}$,
then (\ref{hesse}) has radial
solutions with slow decay rate $\frac{2k}{q-k}$.
\end{theorem}

\begin{proof}
Consider the problem
\begin{equation} \label{odes}
\left \{
\begin{array}{l}
\frac{1}{k}C_{n-1}^{k-1}[(f_r)^k r^{n-k}]_r=r^{n-1} (-f)^q, \quad r>0,\\[3mm]
f(r)<0 \quad and~ f'(r)>0, \quad r>0,\\[3mm]
f(0)=-A, \quad f'(0)=0,
\end{array}
   \right.
\end{equation}
where $A>0$ will be determined later. We prove that this problem has
a bounded entire solution $f(r)$. Then $u(x)=f(|x|)$ is a solution of (\ref{hesse}).

Let $t>0$. Integrating (\ref{odes}) from $0$ to $t$ yields
$$
f_r(t)=(\frac{1}{k}C_{n-1}^{k-1})^{-\frac{1}{k}}
[t^{k-n}\int_0^t r^{n}(-f(r))^q\frac{dr}{r}]^{\frac{1}{k}}.
$$
For $s \geq 0$, integrating from $s$ to $R$ again, we get
\begin{equation}\label{kaoshi}
f(R)=f(s)+(\frac{1}{k}C_{n-1}^{k-1})^{-\frac{1}{k}} \int_s^R
[t^{k-n}\int_0^t r^{n}(-f(r))^q\frac{dr}{r}]^{\frac{1}{k}}dt.
\end{equation}

We claim that $\inf_{r \geq 0}[-f(r)]=0$ as long as $f$ is an entire
positive solution of (\ref{odes}). Otherwise, there exists $c_*>0$
such that $-f(r) \geq c_*$ for $r \geq 0$. Therefore, (\ref{kaoshi})
with $s=0$ shows
$$
f(R) \geq (\frac{kc_*^q}{C_{n-1}^{k-1}})^{\frac{1}{k}}\frac{R^2}{2}-A.
$$
Thus, we can find some $R_A$ such that $f(R_A)=0$. This shows
that (\ref{odes}) has no entire positive solution.

By virtue of $f'(r)>0$ and $\inf_{r \geq 0}[-f(r)]=0$,
we know that the global solution $f(r)$
should be bounded and increasing to zero when $r \to \infty$.
In the following, we construct the function $f(r)$ with the slow decay rate
$\frac{2k}{q-k}$.

First, (\ref{odes}) admits a local negative solution by the standard argument.
Namely, for each $A>0$, we can find $R^A>0$ such that the solution $f(r)<0$ as
$r \in (0,R^A)$. We claim $f(R^A)<0$.

Otherwise, $f(R^A)=0$. Thus, $f(r)$ solves the two point boundary value problem
$$
\left \{
\begin{array}{l}
\frac{1}{k}C_{n-1}^{k-1}[(f_r)^k r^{n-k}]_r=r^{n-1} (-f)^q, \quad r\in (0,R^A),\\
f(0)=-A, \quad f(R^A)=0.
\end{array}
   \right.
$$
Namely, $u(x)=f(|x|)$ is a classical solution of
$$
\left \{
\begin{array}{l}
F_k(D^2u)=(-u)^q \quad on \quad B_{R^A}(0),\\
u|_{\partial B_{R^A}(0)}=0, \quad u<0 \quad on ~~B_{R^A}(0).
\end{array}
   \right.
$$
It contradicts with the nonexistence result in the supercritical case (cf.
\cite{Tso1}).

Thus, we can extend the solution towards right in succession and hence
obtain an entire solution for each $A$.

Next, by the shooting method, it is easy to find suitable $A>0$ such that (\ref{odes}) has
solution $f_A(r)<0$ on $[0,1]$ satisfying $f_A(0)=-A$ and $f_A(1)=-C_A$. Here
$$
C_A=[\frac{C_{n-1}^{k-1}}{k}(n-\frac{2qk}{q-k})]^{\frac{1}{q-k}}
(\frac{q-k}{2k})^{\frac{k}{k-q}}.
$$
By virtue of (\ref{SERRIN}), we see $C_A>0$.

Finally, we claim that
$
f(R)=-C_AR^{-\frac{2k}{q-k}}
$
solves (\ref{odes}) as $r>1$.
It is sufficient to verify (\ref{kaoshi}) with $s=1$:
$$
f(R)=f(1)+(\frac{1}{k}C_{n-1}^{k-1})^{-\frac{1}{k}} \int_1^R
[t^{k-n}\int_0^t r^{n}(-f(r))^q\frac{dr}{r}]^{\frac{1}{k}}dt.
$$

In fact, (\ref{SERRIN}) shows that $n-\frac{2qk}{q-k}>0$.
By a calculation, we get the right hand side of the integral equation
above
$$\begin{array}{ll}
&\quad (\displaystyle\frac{kC_A^q}{C_{n-1}^{k-1}})^{\frac{1}{k}}
\displaystyle\int_1^R(t^{k-n}\int_0^t
s^{n-\frac{2qk}{q-k}}\frac{ds}{s})^{\frac{1}{k}}dt-C_A\\[3mm]
&=(\displaystyle\frac{kC_A^q}{C_{n-1}^{k-1}})^{\frac{1}{k}}
(n-\frac{2qk}{q-k})^{-\frac{1}{k}}
\int_1^R t^{2-\frac{2q}{q-k}}\frac{dt}{t}-C_A\\[3mm]
&=(\displaystyle\frac{kC_A^q}{C_{n-1}^{k-1}})^{\frac{1}{k}}
(n-\frac{2qk}{q-k})^{-\frac{1}{k}}\frac{q-k}{2k}
(1-R^{-\frac{2k}{q-k}})-C_A\\[3mm]
&=-C_AR^{-\frac{2k}{q-k}}.
\end{array}
$$
It is the left hand side.

Thus, we find a radial solution of (\ref{hesse}) in the supercritical case
$$
u(x)=
\left \{
\begin{array}{lll}
&f_A(|x|), \quad &for \quad r \in [0,1];\\
&-C_A|x|^{-\frac{2k}{q-k}}, \quad &for \quad r \geq 1.
\end{array}
   \right.
$$
Theorem \ref{th4.10} is proved.
\end{proof}

\paragraph{Acknowledgements} The author is very grateful to
Dr. Xingdong Tang for providing the reference \cite{IMN},
and many fruitful discussions.


\end{document}